%% file: CartanAffine.tex
\documentclass[10pt]{amsart}
\allowdisplaybreaks

\usepackage{amsfonts}
\usepackage{amsmath}
\usepackage{amssymb}
\usepackage{amsthm}
\usepackage{amsbsy}
\usepackage{xcolor}
\usepackage{graphicx} \usepackage{enumerate} \usepackage{multicol}
\usepackage{mathrsfs} \usepackage[all,cmtip]{xy}
\usepackage[color=blue!20]{todonotes}

\newcounter{dummy} \numberwithin{dummy}{section}
\newtheorem{theorem}[dummy]{Theorem}
\newtheorem{corollary}[dummy]{Corollary}
\newtheorem{lemma}[dummy]{Lemma}
\newtheorem{definition}[dummy]{Definition}
\newtheorem{proposition}[dummy]{Proposition}
\theoremstyle{remark}
\newtheorem{remark}[dummy]{Remark}
\newtheorem{example}[dummy]{Example}

\newcommand{\calA}{\mathcal{A}}
\newcommand{\calL}{\mathcal{L}}
\newcommand{\calH}{\mathcal{H}}

\newcommand{\scrF}{\mathscr{F}}

\newcommand{\frakg}{\mathfrak{g}}

\DeclareMathOperator{\Ann}{Ann}

\DeclareMathOperator{\End}{End}

\DeclareMathOperator{\id}{id}

\DeclareMathOperator{\pr}{pr}
\DeclareMathOperator{\rank}{rank}
\DeclareMathOperator{\spn}{span}
\DeclareMathOperator{\tr}{tr}

\newcommand{\ptr}{/ \! /}

\DeclareMathOperator{\Hol}{Hol}

\DeclareMathOperator{\Ad}{Ad}

\DeclareMathOperator{\T}{T}
\DeclareMathOperator{\K}{K}
\DeclareMathOperator{\free}{free}
\DeclareMathOperator{\gr}{gr}
\DeclareMathOperator{\Gr}{Gr}
\DeclareMathOperator{\Iso}{Iso}
\DeclareMathOperator{\iso}{iso}

\newcommand{\Lbra}{[ \![}
\newcommand{\Rbra}{] \!]}

\numberwithin{equation}{section}

\title[Canonical connections on sub-Riemannian manifolds]{Canonical connections on sub-Riemannian manifolds with constant symbol}
\author[E.~Grong]{Erlend Grong}

\address{University of Bergen, Department of Mathematics, P.O.~Box 7803, 5020 Bergen, Norway}
  \email{erlend.grong@uib.com}

\subjclass[2010]{53C17, 17B70, 58A15}
\keywords{Cartan connection, sub-Riemannian manifolds of constant symbol, graded manifolds, compatible affine connections, contact manifolds, $(2,3,5)$-distributions}

\thanks{The author is supported by the grants Research Council of Norway (project number 249980/F20) and GeoProCo from the Trond Mohn Foundation - Grant TMS2021STG02 (GeoProCo).}

\begin{document}

\begin{abstract}
As a tool to address the equivalence problem in sub-Riemannian geometry, we introduce a canonical choice of grading and compatible affine connection, available on any sub-Riemannian manifold with constant symbol. We completely compute these structures for contact manifolds of constant symbol, including the cases where the connections of Tanaka-Webster-Tanno are not defined. We also give an original intrinsic grading on sub-Riemannian (2,3,5)-manifolds, and use this to present the first flatness theorem in this setting.
\end{abstract}

\maketitle


\input{Introduction}
\input{DefSR}
\input{Carnot}
\input{NonHolT}

\input{Graded}

\input{Morimoto}

\input{Contact}

\input{Cartan235}

\bibliographystyle{habbrv}
\bibliography{Bibliography}

\end{document}

%% file: Introduction.tex
\section{Introduction}
Ever since Strichartz's introduction of sub-Riemannian manifolds \cite{Str86} in 1984, based on earlier works in  e.g. \cite{Bro82,Gav77}, there has been relatively few results related to the equivalence problem of these geometric objects. The known material can be found in \cite{FaGo95,FaGo96,Hug95,BFG99,AgBa12,Alm14}, but a general approach to such a problem has proved quite challenging. While any Riemannian manifold at any point will have the Euclidean space as its metric tangent cone, the generic infinitesimal model of a sub-Riemannian manifold belongs to a large class of spaces called \emph{Carnot groups} \cite{Mit85,Bel96,Gro96}, with each case having their own special considerations. Furthermore, it has been not been clear what the analogue to the main tools of the equivalence problem in Riemannian geometry, Levi-Civita connection and its curvature. See e.g. the Cartan-Ambrose-Hicks theorem~\cite[Chapter 1.12]{ChEb75} and result in \cite{Sin60,PTV96} for examples of works on the equivalence problem of Riemannian manifolds.

A key development in form of a canonical choice of connection was presented by T. Morimoto in his 2008 paper \cite{Mor08}, based on his work in \cite{Mor93}. Using Cartan geometry, Morimoto proved that sub-Riemannian manifolds with \emph{constant symbol}, i.e. with the tangent cones at all points being isometric, have a canonical choice of Cartan connection on its nonholonomic frame bundle. In particular, it is under-appreciated how this result gives us a flatness theorem for sub-Riemannian manifolds, which to our knowledge previously only existed for sub-Riemannian manifolds whose symbol was the standard Heisenberg group. For some result applying Morimoto's formalism to sub-Riemannian manifolds, see \cite{AMS19,BHM20}.

We want to make the theory of Morimoto more explicit and available by translating the result from Cartan geometry to the formalism of a grading of the tangent bundle and a compatible affine connection. In particular, using the theory of selectors developed in \cite{CGJK19}, we give an explicit description of Morimoto's connection. We will show the merits of our method by applying the result to contact manifolds and $(2,3,5)$-manifolds. We emphasize that in the case of contact manifolds, we are consider all cases where we have a constant symbol; not just the case where the sub-Riemannian metric is defined from the contact form as with connections of Tanaka, Webster and Tanno \cite{Tan76,Web78,Tan89}.

The structure of the paper is as follows. In Section~\ref{sec:DefSR} we will introduce the basic definitions of sub-Riemannian manifolds, including theory related to compatible affine connections. Section~\ref{sec:Carnot} will introduce Carnot groups and algebras, and their isometry groups. In Section~\ref{sec:ConstSym} we define the symbol of a sub-Riemannian manifold at a point and the extra structure that exists for manifolds with constant symbol. We also introduce the important notion of \emph{strongly compatible connections} on the nonholonomic tangent bundle and show that these only exist in the case of constant symbols. In Section~\ref{sec:Grading}, we work with sub-Riemannian manifolds whose tangent bundle is graded in a way compatible with the growth of the flag $E = E^{-1} \subseteq E^{-2} \subseteq \cdots \subseteq E^{-s}$ produced by an increasing number of brackets of the sections of the horizontal bundle of $E$. We finally use this formalism to write Morimoto's canonical choice of Cartan connection as a grading and affine connections in Section~\ref{sec:Morimoto}. The final conditions are presented in Theorem~\ref{th:MorimotoGC}.

In the final two sections, we apply our result to two different examples. Further examples can be found in later paper \cite{Gro22b}. In Section~\ref{sec:Contact} we work with contact manifolds whose constant symbol will be the Heisenberg algebra, but not necessarily with the standard metric. Surprisingly, the grading given from Morimoto's theory is not necessarily the one given by the Reeb vector field, nor does its connection  correspond to the connections of Tanaka-Webster-Tanno in the cases when they are defined. In the final section, Section~\ref{sec:235}, we introduce a completely original canonical grading for sub-Riemannian manifolds with growth vectors $(2,3,5)$. We then show how this grading and affine connection corresponding to Morimoto's normalization condition for Cartan connections. In order to present an explicit result, we include this flatness theorem proved in Section~\ref{sec:Structure235}, which, although complicated, can be explicitly computed.

Let $(M,E,g)$ be any sub-Riemannian manifold with growth vector $(2,3,5)$. Choose a local orthonormal basis $X_1$, $X_2$ be any local orthonormal basis of $E$ defined on a neighborhood $U$. Introduce a basis $X_3 = [X_1, X_2]$, $X_4 = [X_1, X_3]$ and $X_5 = [X_2, X_3]$ with $[X_i, X_j] = \sum_{k=1}^5 c_{ij}^k X_k$. Introduce a (local) decomposition $TM  = E \oplus V^{-2} \oplus V^{-3} = E \oplus \spn \{ Z \} \spn \{ Y_1, Y_2\}$ with
\begin{align*}
Z & = X_3 + ( c_{23}^3+c_{24}^4 + c_{25}^5) X_1 - ( c_{13}^3+c_{14}^4 + c_{15}^5) X_2, \\
Y_1 & = X_4 - (c_{14}^4 + c_{15}^5) X_3 + (c_{24}^3 - X_2(c_{14}^4 + c_{15}^5) + c_{24}^4 (c_{14}^4 + c_{15}^5) + c_{24}^5(c_{24}^4 + c_{25}^5)) X_1\\
& \qquad -  (c_{14}^3 - X_1(c_{14}^4 + c_{15}^5) + c_{14}^4 (c_{14}^4 + c_{15}^5) + c_{14}^5(c_{24}^4 + c_{25}^5)) X_2, \\
Y_2 & = X_5 - (c_{24}^4 + c_{25}^5) X_3 + (c_{25}^3 - X_2(c_{25}^4 + c_{25}^5) + c_{25}^4 (c_{14}^4 + c_{15}^5) + c_{25}^5(c_{24}^4 + c_{25}^5)) X_1 \\
& \qquad -  (c_{15}^3 - X_1(c_{25}^4 + c_{25}^5) + c_{15}^4 (c_{14}^4 + c_{15}^5) + c_{15}^5(c_{24}^4 + c_{25}^5)) X_2,
\end{align*}
Define a Riemannian metric $\bar{g}$ be the Riemannnian metric making $X_1, X_2, Z, Y_1, Y_2$ into an orthonormal basis. If $\nabla^{\bar{g}}$ is its Levi-Civita connection, we define a new connection $\nabla$ such that $E$, $V^{-2}$ and $V^{-3}$ are parallel, with $\nabla Z = 0$ and with
\begin{align*}
\langle \nabla_{X_i} X_j, X_k \rangle_{\bar{g}} & =\langle \nabla_{X_i} Y_j, Y_k \rangle_{\bar{g}} = \langle \nabla_{X_i}^{\bar{g}} X_j, X_k \rangle_{\bar{g}}, \\
\langle \nabla_{Z} X_j, X_k \rangle_{\bar{g}} & = \langle\nabla_{Z} Y_j, Y_k \rangle_{\bar{g}} = \langle [Z,X_j], X_k \rangle_{\bar{g}} + \frac{1}{2} (\mathcal{L}_Z \bar{g})(X_j, X_k), \\
\langle \nabla_{Y_i} X_j, X_k \rangle_{\bar{g}} & =\langle \nabla_{Y_i} Y_j, Y_k \rangle_{\bar{g}} = \langle [Y_i,X_j], X_k \rangle_{\bar{g}} + \frac{1}{2} (\mathcal{L}_{Y_i} \bar{g})(X_j, X_k).
\end{align*}

\begin{theorem}[Flatness theorem for $(2,3,5)$-manifolds] \label{th:Flat235} 
The metric $\bar{g}$ and the connection $\nabla$ does not depend on the choice of local basis $X_1,X_2$ and can be defined globally. $(M,E,g)$ is locally isometric to the Cartan nilpotent group if and only if the curvature $R$ of $\nabla$ vanishes and the only non-zero parts of the torsion $T$ of $\nabla$ are given by
$$T(X_2, X_1) = Z, \qquad T(Z, X_j) = Y_j.$$
\end{theorem}

\paragraph{\it Acknowledgement:} We thank Boris Kruglikov and Dennis The for helpful questions and for pointing out the reference \cite{Sag08}.

%% file: DefSR.tex
\section{Sub-Riemmannian manifolds} \label{sec:DefSR}
Let $M$ be a connected manifold. A sub-Riemannian structure is a pair $(E,g)$ where $E$ is a subbundle of the tangent bundle $TM$ and $g = \langle \, \cdot \, , \, \cdot \, \rangle_g$ is a metric tensor defined only on $E$. We call $E$ \emph{the horizontal bundle}. For the rest of the paper, we will assume that $E$ is \emph{bracket-generating}, meaning that the sections of $E$ and their iterated brackets span all of $TM$. With this assumption, we can make $(M, E,g)$ into a metric space with the following construction. An absolutely continuous curve $\gamma:[a,b] \to M$ is called \emph{horizontal} if $\dot \gamma(t) \in E_{\gamma(t)}$ for almost every~$t$. For such a curve, we can define its length as
$$\mathrm{Length}(\gamma) = \int_a^b \langle \dot \gamma , \dot \gamma \rangle_g^{1/2}(t) \, dt.$$
We can then define \emph{the Carnot-Carath\'eodory distance} $d_{cc}(x,y)$ between the points $x,y \in M$ as the infimum of the lengths of all horizontal curves connecting them. When $E$ is bracket-generating, this distance is always finite and it induces the same topology as the manifold topology.

In this paper we will make the additional assumption that $E$ is \emph{equiregular}, meaning that for any $k=0,1, \dots, $
$$E^{-k} = \spn \{ [X_1, \cdots, [X_{j-2} ,[X_{j-1},X_j]] \cdots]|_x \, : j=0,\cdots, k, X_i \in \Gamma(E), x \in M\},$$
is not only a subset, but a subbundle of $TM$. In particular, the rank of $E_x^{-k}$ does not depend on $x$. In the above expression, for the cases $j =0$ and $j =1$, we interpret the bracket $[X_1, \cdots, [X_{j-2} ,[X_{j-1},X_j]] \cdots]$ as respectively $0$ and $X_1$. We have a corresponding flag of subbundles,
\begin{equation} \label{Flag} E^0 = 0 \subsetneq E^{-1} = E \subsetneq E^{-2} \subsetneq \cdots \subsetneq E^{-s}  = TM.\end{equation}
The minimal integer $s$ such that $E^{-s} = TM$ is called \emph{the step of $E$} and the collection of the ranks $\mathfrak{G} = (\rank E^{-1}, \dots, \rank E^{-s})$ is called \emph{the growth vector} of $(M,E,g)$.
For more on sub-Riemannian manifolds, see \cite{Mon02,ABB20}.

We make the following definition related to affine connections and compatibility with a sub-Riemannian structure $(E,g)$.
\begin{definition}
Let $(E,g)$ be a sub-Riemannian structure on a manifold $M$.
\begin{enumerate}[\rm (i)]
\item Let $\nabla$ be an affine connection on $TM$. The subbundle $E$ is called parallel with respect to $\nabla$ if $\nabla_Y X |_x  \in E_x$ for any $Y \in \Gamma(TM)$, $X \in \Gamma(E)$, $x\in M$. Note that if $E$ is parallel, we obtain an affine connection on $E$ by restriction $\nabla|E$ of $\nabla$.
\item An affine connection $\nabla^E$ on $E$ is said to be compatible with $g$ if
$$Y \langle X_1, X_2 \rangle_g = \langle \nabla_Y^E X_1, X_2 \rangle_g + \langle X_1, \nabla_Y^E X_2 \rangle_g,$$
for any $Y \in \Gamma(TM)$, $X_1, X_2 \in \Gamma(E)$.
\item An affine connection $\nabla$ on $TM$ is compatible with $(E,g)$ if $E$ is parallel and $\nabla|E$ is compatible with $g$.
\end{enumerate}
\end{definition}

Without getting into the general theory of length minimizers in sub-Riemannian manifolds, for which we again refer to  \cite{Mon02,ABB20}, we mention the following relationship between them and compatible connections found in \cite{GoGr17,Gro20}. Recall that the torsion $T$ of a connection $\nabla$ is defined as
$$T(X, Y) = \nabla_X Y - \nabla_Y X - [X, Y].$$
\begin{proposition}
Let $\nabla$ be any affine connection on $TM$ compatible with $(E,g)$ and with torsion $T$. Let $\gamma: [a,b] \to M$ be a length minimizer of the distance $d_{cc}$ between the points $\gamma(a)$ and $\gamma(b)$. Then at least one of the following holds.
\begin{enumerate}[\rm (i)]
\item (Normal geodesic) There is a one-form $\lambda(t)$ along $\gamma(t)$ satisfying for any $u \in E_{\gamma(t)}$ and $w \in T_{\gamma(t)}M$,
$$(\nabla_{\dot \gamma} \lambda(t)) (w) = - \lambda(t) (T(\dot \gamma(t), w)), \qquad \lambda(t)(u) = \langle \dot \gamma(t), u \rangle_g.$$
\item (Abnormal curves) There is a one-form $\lambda(t)$ along $\gamma(t)$ satisfying for any $u \in E_{\gamma(t)}$ and $w \in T_{\gamma(t)}M$,
$$(\nabla_{\dot \gamma} \lambda(t)) (w) = - \lambda(t) (T(\dot \gamma(t), w)), \qquad \lambda(t)(u) = 0.$$
\end{enumerate}
Furthermore, curves satisfying {\rm (i)} above are always smooth and local length minimizers.
\end{proposition}

\begin{remark}
If $E$ is a subbundle of $TM$ with rank greater or equal to $2$, then it will generically be bracket-generating in the sense of \cite[Proposition 2]{Mon93}. Furthermore, any bracket-generating subbundle $E$ will be equiregular outside of a closed set with empty interior \cite[Section 2.1.2, p. 21]{Jea14}. \end{remark}

\begin{remark}
Here and in the next section, we use $E^k$ with $k$ negative. Similarly, we will in Section~\ref{sec:Carnot} use negative grading for Carnot algebras. This convention is used to make our notation coincide with the one used when dealing with Tanaka prolongations, but as the prolongations of Carnot algebras have no positive part \cite[Proposition~1]{Mor08}, we will not encounter positively graded subspaces.
\end{remark}

%% file: Carnot.tex
\section{Carnot groups and algebras} \label{sec:Carnot}
\subsection{Definitions}
\emph{A stratified Lie algebra} is a graded Lie algebra
$$\mathfrak{g}_- = \mathfrak{g}_{-s} \oplus \cdots \oplus \mathfrak{g}_{-1}, \quad \text{ such that } \quad [\mathfrak{g}_{-1}, \mathfrak{g}_{-k}] = \mathfrak{g}_{-k-1},$$
for $k = 1, \dots, s-1$ and with $[\mathfrak{g}_{-1}, \mathfrak{g}_{-s}]=0$. In particular, $\mathfrak{g}_-$ is nilpotent. If in addition, $\mathfrak{g}_{-1}$ is furnished with an inner product $\langle \, \cdot \, , \, \cdot \, \rangle_{\mathfrak{g}_{-1}}$, then $\mathfrak{g}_-$ is called \emph{a Carnot algebra}. Let $G_-$ be the corresponding simply connected Lie group and define a sub-Riemannian structure $(E,g)$ on $G_-$ by left translation of $\mathfrak{g}_{-1}$ and its inner product. We say that $(G_-,E,g)$ is \emph{a Carnot group} which will be a sub-Riemannian manifold with an equiregular horizontal bundle. For examples, see Sections~\ref{sec:Heisenberg} and~\ref{sec:CarnotNil}.

\subsection{Isometries} \label{sec:Isometries}
We introduce some  general facts and notation related to isometry groups of Carnot groups and their Lie algebras. For details regarding this information we refer to~\cite{LDO16,Gro16b}. Any isometry of $(G_-, E,g)$ is a composition of a left translation and a group automorphism. More precisely, a group automorphism $\Psi:G_- \to G_-$ is a sub-Riemannian isometry if the corresponding Lie algebra automorphism $\psi = \Psi_{*,1} : \mathfrak{g}_- \to \mathfrak{g}_-$ maps $\mathfrak{g}_{-1}$ onto~$\mathfrak{g}_{-1}$ isometrically. For this reason, we introduce the following definitions.
\begin{definition}
Let $\mathfrak{g}_- = \mathfrak{g}_{-s} \oplus \cdots \oplus \mathfrak{g}_{-1}$ and $\tilde {\mathfrak{g}}_- = \tilde {\mathfrak{g}}_{-s} \oplus \cdots \oplus \tilde {\mathfrak{g}}_{-1}$ be two Carnot algebras.
\begin{enumerate}[\rm (i)]
\item A linear map $\psi: \mathfrak{g}_- \to \tilde {\mathfrak{g}}_-$ will be called an isometry if $\psi$ is a Lie algebra isomorphism that maps $\mathfrak{g}_{-1}$ onto~$\tilde{\mathfrak{g}}_{-1}$ isometrically. We write $\Iso(\mathfrak{g}_-, \tilde {\mathfrak{g}}_-)$ for the space of all such maps.
\item We say that $\mathfrak{g}_-$ and $\tilde {\mathfrak{g}}_-$ are isometric if there is an isometry connecting them.
\end{enumerate}
\end{definition}
Observe that isometries have degree zero from the definition of Carnot algebras. The group $G_0 = \Iso(\mathfrak{g}_-) = \Iso(\mathfrak{g}_-, \mathfrak{g}_-)$ is a Lie group with Lie algebra $\mathfrak{g}_0 = \iso(\mathfrak{g}_-)$ consisting of derivations of degree zero such that their restriction to $\mathfrak{g}_{-1}$ are skew-symmetric. In other word, $\mathfrak{g}_0$ consists of all maps $D: \mathfrak{g}_- \to \mathfrak{g}_-$ such that $D(\mathfrak{g}_{-k}) \subseteq \mathfrak{g}_{-k}$ and such that for any $B, C \in \mathfrak{g}_-$ and $A \in \mathfrak{g}_{-1}$,
$$D[B,C] = [DB, C] + [B, DC], \qquad \langle DA, A \rangle_{\mathfrak{g}_{-1}} = 0.$$

We can define an extended Lie algebra $\mathfrak{g} = \mathfrak{g}_- \oplus \mathfrak{g}_0 = \mathfrak{g}_- \oplus \iso(\mathfrak{g}_-)$ with brackets,
$$[ A +D_1 ,  B+ D_2] = [A,B] + D_1 B- D_2 A + [D_1, D_2] , \qquad D_1, D_2 \in \mathfrak{g}_0, A,B \in \mathfrak{g}_-.$$
This will be the Lie algebra of the isometry group of $(G_-, E, g)$.

\subsection{The free nilpotent algebra and induced inner products on Carnot algebras} \label{sec:Free}
Recall first the following definition. If $q:W \to V$ is a surjective linear map between two inner product spaces, we say that it is a \emph{linear submetry} if $q|_{(\ker q)^\perp}$ is an isometry from $(\ker q)^\perp$ onto $V$. We observe that if $W$ is an inner product space and $q: W \to V$ is a surjective linear map to a vector space, there is a unique inner product on $V$ making $q$ into a linear submetry.

Let $W$ be a vector space and define $\T(W,s) = \oplus_{j=1}^s W^{\otimes j}$ as the truncated tensor algebra up to step $s$. We consider $\T(W,s)$ as the an algebra with product $(A,B) \mapsto A \otimes B$ using the convention that tensor products of more than $s$ elements in $W$ vanish. We define a subalgebra $\K(W,s)$ of elements
$$\K(W,s) = \spn \{ A \otimes B + B \otimes A, A \otimes B \otimes C + B \otimes C \otimes A + C \otimes A \otimes B \, : \, A,B,C \in \T(W,s)\}. $$
and a submersion
\begin{equation} \label{kmap}
k: \T(W,s) \to \free(W,s) := \T(W,s)/\K(W,s),
\end{equation}
We furnish $\free(W,s)$ with a Lie algebra structure given by
$$[k(A) , k(B)] = k(A \otimes B - B \otimes A).$$
This algebra will be a stratified Lie algebra
$$\mathfrak{f}_- = \free(W,s) = \mathfrak{f}_{-s} \oplus \cdots \oplus \mathfrak{f}_{-1}, \qquad \text{ with } \mathfrak{f}_{-j} = k(W^{\otimes j});$$
so in particular $\mathfrak{f}_{-1} = W$. The Lie algebra $\mathfrak{f}_- = \free(W,s)$ is called \emph{the free nilpotent algebra of step $s$ generated by $W$}.
If $W$ is an inner product space, then $\mathfrak{f} =\free(W,s)$ has the structure of a Carnot algebra. If we define $F_0 = \Iso(\mathfrak{f})$, then there is a bijection between linear isometries $q: W\to W = \mathfrak{f}_{-1}$ and elements $\psi_q \in F_0$ given by $\psi_q|_{\mathfrak{f}_{-1}} = q$.

Let $\mathfrak{g}_- =  \mathfrak{g}_{-s} \oplus \cdots \oplus \mathfrak{g}_{-1}$ be any stratified Lie algebra. Then there is a surjective Lie algebra homomorphism
\begin{equation} \label{Pmap}
P: \free(\mathfrak{g}_{-1},s) \to \mathfrak{g}_-
\end{equation}
given by dividing out the additional relations. It follows that $\mathfrak{g}_-$ is uniquely determined by the ideal $\mathfrak{a} =\ker P$. Furthermore, the Lie group $G_0 = \Iso(\mathfrak{g}_- )$ is isomorphic to the Lie group
$$\{ \psi \in F_0 \, : \, \psi(\mathfrak{a}) \subseteq \mathfrak{a}\}.$$

Assume that $\mathfrak{g}_- = \mathfrak{g}_{-s} \oplus \cdots \oplus \mathfrak{g}_{-1}$ has the structure of a Carnot algebra. We can then define the inner product in an iterative way as follows.
\begin{enumerate}[$\bullet$]
\item The grading of $\mathfrak{g}_{-}$ induces a grading on $\wedge^2 \mathfrak{g}_{-}$. Explicitly,
$$(\wedge^2 \frakg_{-})_{-k} = \spn \{ A_1 \wedge A_2 \, : \, A_1\in \frakg_{-i}, A_2 \in \frakg_{-j}, i+j =k\}.$$
Observe that evaluating elements in $(\wedge^2 \frakg_-)_{-k}$ under the Lie bracket gives a surjective map
\begin{equation} \label{BracketSubmetry} [\cdot, \cdot]: (\wedge^2 \frakg_-)_{-k} \to \frakg_{-k}.\end{equation}
Notice also that if $\frakg_{\geq -k} = \frakg_{-k} \oplus \cdots \oplus \frakg_{-1}$,
then $(\wedge^2 \frakg)_{-k} \subseteq \wedge^2 \frakg_{\geq -k+1}$.
\item From the inner product on $\frakg_{-1}$, define an inner product on $\wedge^2 \frakg_{-1} = (\wedge^2 \frakg_{-})_{-2}$ such that if $A_1, \dots, A_n$ is any orthonormal basis $\frakg_{-1}$, then $A_i \wedge A_j$, $i <j$ form an orthonormal basis of $\wedge \frakg_{-1}$. As the Lie brackets are a surjective map from the latter space to $\frakg_{-2}$, there is a unique inner product on $\frakg_{-2}$ that makes the mentioned map a submetry.
\item Repeating this procedure, if we have an inner product defined on $\frakg_{\geq -k+1}$, then we can induce an inner product on $(\wedge^2 \frakg_-)_{-k} \subseteq \wedge^2 \frakg_{-k+1}$, and give $\frakg_{-k}$ the inner product such that \eqref{BracketSubmetry} a submetry.
\end{enumerate}
We also have an an induced inner product on $\mathfrak{g}_0 = \iso(\mathfrak{g}_- )$. Finally, any $D \in \frakg_0$ can be considered as $D \in \End \frakg_{-} = \frakg_{-}^* \otimes \frakg_{-}$ and define
$$\langle D_1, D_2 \rangle_{\mathfrak{g}_0} = \tr_{\mathfrak{g}_-}\langle  D_1 \times, D_2 \times \rangle_{\mathfrak{g}_-}.$$
\begin{remark} Let $D \in \frakg_0$ act on elements in $\wedge^2 \frakg_-$ by $D \cdot (A \wedge B) = DA \wedge B + A \wedge DB$, $A,B \in \frakg_-$. Since any $D \in \frakg_0$ is a Lie algebra derivation, it will preserve $\ker [\cdot, \cdot]$ of elements in the wedge product that vanish under the Lie bracket. Using that by our construction of the inner product on $\frakg_{-1}$, since
$$\langle [A_1,A_2], [B_1, B_2]\rangle = \left\langle \pr_{(\ker [\cdot, \cdot])^{\perp}} A_1 \wedge A_2, \pr_{(\ker [\cdot, \cdot])^{\perp}} B_1 \wedge B_2  \right\rangle$$
for $A_1 \wedge A_2, B_1 \wedge B_2 \in (\wedge^2 \frakg_-)_{-k}$, by an iterative argument, $D$ is also skew-symmetric on $\frakg_-$ and preserves $(\ker [\cdot, \cdot])^\perp$.
\end{remark}

\begin{remark}
The above construction also make sense when $W$ is a vector bundle rather than a vector space. We will take advantage of this in Section~\ref{sec:CompNonHol}.
\end{remark}

%% file: NonHolT.tex
\section{Nonholonomic tangent bundle and constant symbol} \label{sec:ConstSym}

\subsection{Sub-Riemannian manifolds with constant symbols}
Let $(M,E,g)$ be a sub-Riemannian manifold. Assume that $E$ is bracket-generating and equiregular. We define \emph{the nonholonomic tangent bundle} of $(M,E,g)$ by
$$\gr E = E^{-s}/E^{-s+1} \oplus \cdots \oplus E^{-2}/E \oplus E.$$
For each $x \in M$, $\gr_x ;= \gr E_x$ can be given the the structure of a Carnot algebra, with grading $\gr_{x,-k} := (\gr_x)_{-k} = E^{-k}_x/E^{-k+1}_x$ and with Lie brackets $\Lbra \, \cdot \, , \, \cdot \, \Rbra$ defined by
$$\Lbra X_x \bmod E^{-i}_x , Y_x \bmod E^{-j}_x \Rbra = [X, Y]|_x \bmod E^{-i-j}_x.$$
Here, $[X,Y]$ denotes the commutator bracket of the vector fields $X$ and $Y$ extending the vectors $X_x$ and $Y_x$. We observe that this bracket is well defined as the right hand side does not depend on the chosen extensions. Since $\gr_{x,-1} = E_x$ has an inner product, we have a Cartan algebra structure on $\gr_x$. We say that $\gr_x$ with its Cartan algebra structure is the \emph{symbol} of $(M,E,g)$ at $x \in M$.
\begin{definition}
Let $(M,E,g)$ be a sub-Riemannian manifold with $E$ bracket-generating and equiregular and let $\mathfrak{g}_-$ be a Cartan algebra. We say that $(M,E,g)$ has constant symbol $\mathfrak{g}_-$ if $\gr_x$ is isometric to $\mathfrak{g}_-$ for any $x \in M$.
\end{definition}
If $(M,E,g)$ has constant symbol $\mathfrak{g}_-$, we can then define \emph{nonholonomic frame bundle} as the principal bundle
$$G_0 \to \mathscr{F} \to M,$$
where $\mathscr{F}_x = \Iso(\mathfrak{g}_-, \gr_x)$ for any $x \in M$, $G_0 = \Iso(\mathfrak{g}_-)$ and $G_0$ acts on $\mathscr{F}_x$ by composition on the right.

\begin{remark}
Let $\Gr_x$ be the Carnot group corresponding to $\gr_x$ with its sub-Riemannian structure. Recall the definition of the Carnot Caratheodory metric $d_{cc}$ in Section~\ref{sec:DefSR}. The space $\Gr_x$ is then the tangent cone of the metric space $(M, d_{cc})$ at $x\in M$, see \cite{Mit85,Bel96,Gro96}. Hence, assuming constant symbol is equivalent to assuming that all of the tangent cones of $(M,d_{cc})$ are isometric.
\end{remark}

\subsection{Compatible connections on the nonholonomic tangent bundle} \label{sec:CompNonHol}
Assume that $E$ is bracket-generating, equiregular and define $\gr = \gr E$.
\begin{definition}
 Let $\nabla^{\gr}$ be an affine connection on $\gr$.
\begin{enumerate}[\rm (i)]
\item We say that $\nabla^{\gr}$ is a compatible connection if $\gr_{-j} = E^{-j}/E^{-j+1}$ is a parallel subbundle for $j =1,\dots, s$ and, furthermore, its restriction to $\gr_{-1} = E$ is compatible with the sub-Riemannian metric $g$.
\item We say that $\nabla^{\gr}$ is strongly compatible if it is compatible and furthermore, for any $Y \in \Gamma(TM)$, $A, B \in \Gamma(\gr)$, we have
\begin{equation} \label{CompBracket} \nabla_Y^{\gr} \Lbra A, B \Rbra = \Lbra \nabla_Y^{\gr} A, B \Rbra + \Lbra A, \nabla_Y^{\gr} B \Rbra.\end{equation}
\end{enumerate}
\end{definition}
Compatible connections will always exist by taking a direct sum of connections on each $\gr_{-j}$, $j=1, \dots,s$, with the restriction that the one used on $\gr_{-1}$ has to be a metric connection. By contrast, strongly compatible only exists in the following special case.
\begin{proposition}
There exists a strongly compatible connection on $\gr$ if and only if $(M,E,g)$ has constant symbol.
\end{proposition}
\begin{proof}
Assume first that $\nabla^{\gr}$ is strongly compatible. For a curve $\gamma:[0,1] \to M$ connecting two arbitrary points $x= \gamma(0)$ and $y = \gamma(1)$, let $\ptr^{\gamma}_t: \gr_x \to\gr_{\gamma(t)}$ denote the parallel transport along the curve. Since the connection is compatible, the map $\ptr^\gamma_{1}: \gr_x \to \gr_y$ preserves the grading and is a linear isometry from $E_x$ to $E_y$. Furthermore, since it is strongly compatible, $\ptr^\gamma_{1}$ will also be a Lie algebra homomorphism. Since it is invertible, it follows that $\ptr_1^\gamma$ is an isometry of Cartan algebras. As $x$ and $y$ were arbitrary, $(M,E,g)$ has constant symbol.

Conversely, if $(M,E,g)$ has constant symbol $\mathfrak{g}_-$, we can define the nonholonomic frame bundle $G_0 \to \mathscr{F} \to M$. By \cite[Theorem~II.2.1]{KoNo63}, any such principal bundle has a choice of a principal connection $\omega$. Consider the representation $\Ad$ of $G_0$ on $\mathfrak{g}_-$, given by
\begin{equation}
\label{Adg-} \Ad(\psi)(A) = \psi A, \qquad \psi \in G_0 = \Iso(\mathfrak{g}_-), A \in \mathfrak{g}_-.\end{equation}
We can then identify the associated bundle $\Ad(\scrF)$ with $\gr(E)$, through the vector bundle isomorphism
$$f \times_{\Ad} A \mapsto f(A), \qquad f\in \scrF, A\in \mathfrak{g}_-.$$ 
Define $\nabla^{\gr}$ as the induced connection on $\gr$ by $\omega$. By definition, parallel transport $\ptr_t^\gamma: \gr_x \to \gr_{\gamma(t)}$ along any curve $\gamma$ will be a Cartan algebra isometry, which is equivalent to being strongly compatible.
\end{proof}

If $(M,E,g)$ is a manifold with constant symbol, we have the following description of strongly compatible connections. Assume that $\nabla^{\gr}$ is a strongly compatible connection on $\gr = \gr E$. Let $\nabla^E = \nabla^{\gr}|E$ denote the restriction of $\nabla^{\gr}$ to $\gr_{-1} = E$. By \eqref{CompBracket} it follows that $\nabla^E$ iteratively determine $\nabla^{\gr}$, but such a connection cannot be arbitrary as the next result shows.
\begin{proposition} \label{prop:ExtendConnection}
Let $\Hol_x(\nabla^E)$ denotes the holonomy group of $\nabla^E$ at $x$. Then the following are equivalent.
\begin{enumerate}[\rm (i)]
\item $\nabla^E$ is the restriction of a strongly compatible connection $\nabla^{\gr}$ on $\gr$.
\item For any $x \in M$, if
$$\hat S_x = \Iso(\gr_x),$$
then $\Hol_x(\nabla^E) \subseteq \hat S_x |_{E}$.
\end{enumerate}
\end{proposition}
\begin{proof}
Let $\nabla^E$ be a compatible connection on $E$. Using the notation from Section~\ref{sec:Free}, we note that $\nabla^E$ induces a connection on the tensor space $\T(E, s)$. We furthermore observe that $\K(E,s)$ is parallel under this connection, giving us an induced connection on $\free(E,s)$ which we denote by $\nabla^{\free}$. By definition, the Lie brackets of $\free(E,s)$ will be parallel with respect to this connection. If we define~$P$ as in \eqref{Pmap}, then we have a strongly compatible connection $\nabla^{\gr}$ on $\gr$ whose restriction to $E$ is $\nabla^E$ if and only if $\nabla^{\gr} P = P \nabla^{\free}$. We can use such a relation to define a connection $\nabla^{\gr}$ if and only if $\ker P$ is parallel with respect to $\nabla^{\free}$. By our comment on self-isometries of Carnot algebras in Section~\ref{sec:Free}, this is the case if and only if $\Hol(\nabla^{\free})|_{E} = \Hol(\nabla^E) \subseteq  \Iso(\gr_x) |_{E}$. 
\end{proof}

Recall the definition of curvature of a connection $\nabla^E$ on $E$ being given by
$R^E(Y_1, Y_2) X = \left[ \nabla_{Y_1}^E , \nabla_{Y_2}^E\right] X - \nabla_{[Y_1, Y_2]}^E X$ with $Y_1, Y_2 \in \Gamma(TM), X\in \Gamma(E)$.
The next result then follow from the Ambrose-Singer theorem \cite{AmSi53}.
\begin{corollary} \label{cor:AS53}
Assume that $M$ is simply connected. Then the following are equivalent.
\begin{enumerate}[\rm (i')]
\item $\nabla^E$ is the restriction of a strongly compatible connection $\nabla^{\gr}$ on $\gr$.
\item For any $x \in M$, if $\hat{\mathfrak{s}}_x = \iso(\gr_x)$,
then $R^E(u,v) \in \hat{\mathfrak{s}}_x|_{E}$ for any $u,v \in T_xM$.
\end{enumerate}

\end{corollary}

\begin{remark} \label{re:Maxg0}
We remark that in the special case when $\mathfrak{g}_0$ is maximal, i.e. in the special case where any skew-symmetric map of $\mathfrak{g}_{-1}$ induces a Lie algebra automorphism of $\mathfrak{g}_-$, then the above result says that any compatible connection on $E$ can be extended uniquely to a strongly compatible connection on $\gr E$.
\end{remark}

%% file: Graded.tex
\section{Graded sub-Riemannian structures} \label{sec:Grading}

\subsection{Equiregular subbundles and grading}
Let $(M, E,g)$ be the sub-Riemannian manifold with $E$ bracket-generating and equiregular of step $s$. Related to such a subbundle, we introduce the following important definition.
\begin{definition}
Let $TM = \oplus_{k=1}^s(TM)_{-k}$, be a negative grading of $TM$. We will call it an $E$-grading if
$$E^{-k} = (TM)_{-k} \oplus \cdots \oplus (TM)_{-1} =   (TM)_{-k} \oplus E^{-k+1}.$$
\end{definition}

We will use $\pr_{-k}$ to denote the associated projection $TM \to (TM)_{-k}$, $k =1, \dots, s$. We then have a corresponding vector bundle isomorphism $I:  T M \to \gr$
$$I: v  \mapsto \oplus_{k=1}^s I_{-k} v = \oplus_{k=1}^s \pr_{-k} v \bmod E^{-k+1}.$$
Observe that for every $k =1, \dots, s$, the restriction of $I_{-k}$ to $E^{-k}$ is independent of grading and that
\begin{equation} \label{IProperty} I_{-k} w = w \bmod E^{-k+1} \qquad \text{ for every $w \in E^{-k}$}.\end{equation}
Conversely, let $I = \oplus_{k=1}^s I_{-k}: TM \to \gr$ be a vector bundle map satisfying \eqref{IProperty}.
Then we obtain an $E$-grading by
$$(TM)_{-k} = \{ w \in TM \, : \, I_{-j} w =0 \bmod E^{-j+1}, k\neq j\}.$$
Hence, we can equivalently define an $E$-grading as a vector bundle map $I:TM \to \gr$ satisfying \eqref{IProperty}. We will use the two points of view interchangeably.

\begin{definition}
If $I$ is an $E$-grading on the sub-Riemannian manifold $(M,E,g)$, we call $(M, E,g, I)$ a graded sub-Riemannian manifold.
\end{definition}

We say that a connection $\nabla$ on $TM$ is \emph{compatible} with the graded sub-Riemannian structure $(E,g,I)$ if it is compatible with $(E,g)$ and each component $(TM)_{-k}$ is parallel. This connection induces a compatible connection $\nabla^{\gr}$ on $\gr E$ by $I\nabla = \nabla^{\gr} I$. We say that $\nabla$ is \emph{strongly compatible} if $\nabla^{\gr}$ is strongly compatible. Equivalently, $\nabla$ is strongly compatible if the following tensor
\begin{equation} \label{BBoldT} \mathbb{T}(v,w) := - I^{-1}\Lbra Iv , Iw \Rbra, \qquad v,w \in TM,\end{equation} 
is parallel. Observe the following.
\begin{lemma} \label{lemma:Torsion}
If $\nabla$ is a compatible connection of $(M, E, g, I)$ with induced connection $\nabla^{\gr}$ on $\gr$, then its torsion $T$ is given by $I T= d^{\nabla^{\gr}} I$.
\end{lemma} 
\begin{proof}
By definition, we have the identity
\begin{align*}
IT(X,Y) &= I\nabla_X Y - I \nabla_YX - I[X,Y] \\
& = \nabla_X^{\gr} I Y - \nabla_Y^{\gr} I X - I[X,Y] =  d^{\nabla^{\gr}}I(X,Y).\qedhere \end{align*} 
\end{proof}
As a consequence of Lemma~\ref{lemma:Torsion}, we have that for $I_{-k} X=0 \bmod E^{-k+1}$ and $I_{-k} Y = 0 \bmod E^{-k+1}$, then
$$I_{-k} T(X,Y) = - I_{-k}[X,Y].$$
In particular, if $X \in \Gamma(E^{-i})$, $Y \in \Gamma(E^{-j})$ and $k> i+j$, then
\begin{equation} \label{TorsionId} I_{-i-j} T(X,Y) =  \mathbb{T} (X, Y),\qquad I_{-k} T(X,Y) = 0.\end{equation}
It follows that $\mathbb{T}$ is the homogeneous part of degree $0$ of $T$.

\subsection{Gradings and selectors} \label{sec:GradingSelector}
We include the following definition from \cite{CGJK19}. Let $E$ be a bracket-generating and equiregular subbundle with canonical flag as in \eqref{Flag}.
\begin{definition}
A map $\chi: TM \to \wedge^2 TM$ is a selector of $E$ if
\begin{enumerate}[\rm (a)]
\item $\chi(E^{-k}) \subseteq \wedge^2 E^{-k+1}$,
\item If $\alpha$ is any one-from vanishing on $E^{-k}$ for any $k \geq 1$, we have that
$$v \mapsto  (\id + \chi^* d)\alpha( v) = \alpha (v)+ d\alpha(\chi(v)),$$
vanishes on $E^{-k-1}$.
\end{enumerate}
\end{definition}
One can consider $\chi$ as a choice of ``anti-Lie bracket'', in the sense that if $Z \in \Gamma(E^{-k})$, then $\chi(Z) = \sum_{j=1}^k X_j \wedge Y_j$ with $X_j, Y_j \in \Gamma(E^{-k+1})$ satisfies
$$\textstyle Z = \sum_{j=1}^k [X_j,Y_j] \mod E^{-k+1}.$$
Any bracket-generating and equiregular subbundle $E$ has at least one selector. If we have a graded sub-Riemannian manifold $(M,E,g,I)$, we will show that we have a canonically associated selector of $E$. By Section~\ref{sec:Free} we have a canonical extension of $g$ to a metric tensor $\bar{g}$ on $\gr = \gr E$ from its fiber-wise Carnot algebra structure. We can induce an inner product on $\wedge^2 \gr$ such that if $A_1, \dots, A_n$ is a local orthonormal basis of $\gr$, then $\{ A_i \wedge A_j \, : i < j \}$ is a local orthonormal basis for $\wedge^2 \gr$. We then define a map $\chi_0: \gr \to \wedge^2 \gr$ by
$$\langle \Lbra A, B \Rbra, C \rangle_{\bar{g}} = \langle A \wedge B,  \chi_0(C)  \rangle_{\bar{g}}, \qquad A,B,C \in \gr.$$
From the way we defined the metric, if $\chi_0(C) = \sum_{i=1}^k A_i \wedge B_i$ then $\sum_{i=1}^k \Lbra A_i, B_i \Rbra = C$ for any $C \in \Lbra \gr, \gr \Rbra$. We finally have a selector $\chi_I$ corresponding to the $E$-grading $I$ given by
\begin{equation*} \label{IToChi} \chi_I = (\wedge^2 I^{-1}) \chi_0 I.\end{equation*}
Alternatively, define a Riemannian metric $g_I = \langle \cdot , \cdot \rangle_I$ on $M$ by $\langle v, w \rangle_I = \langle Iv, I w \rangle_{\bar{g}}$. Then
$$-\langle \mathbb{T}(u, v), w \rangle_I = \langle u \wedge v,  \chi_I(v)  \rangle_{\bar{g}}, \qquad u,v,w \in \gr.$$
\begin{example}[Canonical selector in step 2]
Consider the special case when we have step $s =2$. Then $I_{-2} v = v \bmod E$, while $I_{-1} = \pr_{-1}: TM \to E$ is a projection. The inverse is given by
$$I^{-1} (u \oplus w \bmod E) = u + w -\pr_{-1} w, \qquad u \in E = \gr_{-1}, w \bmod E \in \gr_{-2}.$$
Since $\chi_0$ vanishes on $E$ and has image in $\gr_{-1}$, it follows from \eqref{IToChi} that the selector $\chi_I$ is actually independent of $E$-grading $I$ chosen.
\end{example}

%% file: Morimoto.tex
\section{Morimoto's Cartan connection as a grading and an affine connection} \label{sec:Morimoto}
\subsection{Cartan connections} \label{sec:Tanaka}
We recall the definition of Cartan connections, see, e.g., \cite{Sha97,CaSl09} for more details. Let $\mathfrak{h}$ be a subalgebra of the Lie algebra $\mathfrak{g}$ of codimension~$n$. Let $H \to P \stackrel{\pi}{\to} M$, be a principal bundle with the group acting on the right such that~$M$ has dimension~$n$ and such that $H$ has Lie algebra~$\mathfrak{h}$. Let $\Ad$ be a representation of $H$ on $\mathfrak{g}$ extending the usual adjoint action of $H$ on $\mathfrak{h}$. \emph{A Cartan connection $\varpi$ on~$P$ modeled on $(\mathfrak{g},\mathfrak{h})$} is a $\mathfrak{g}$-valued one form $\varpi: TP \to \mathfrak{g}$ such that
\begin{enumerate}[\rm (i)]
\item For each $p \in P$, $\varpi|_p$ is a linear isomorphism from $T_pP$ to $\mathfrak{g}$.
\item For each $a \in H$, $v \in TP$, $\varpi(v \cdot a) = \Ad(a^{-1}) \varpi(v)$.
\item For every $D \in \mathfrak{h}$, $p \in P$, we have $\varpi(\frac{d}{dt} p \cdot  \exp_H(tD) |_{t=0}) = D$.
\end{enumerate}
The curvature of a Cartan connection is the $\mathfrak{g}$-valued two-form defined by
$$\textstyle K = d\varpi + \frac{1}{2} [\varpi, \varpi].$$
We verify that if $v \in TP$ satisfies $\varpi(v) \in \mathfrak{h}$ then $K(v, \, \cdot \,) = 0$. Hence, the curvature can represented by a smooth function $\kappa: P \to \wedge^2 (\mathfrak{g}/\mathfrak{h})^* \otimes \mathfrak{g}$, given by
$$K(v,w) = \kappa(p)(\varpi(v), \varpi(w)), \qquad v, w\in T_p M.$$

Assume that $\mathfrak{g}$ is an reductive $\mathfrak{h}$-module, i.e. assume that $\mathfrak{g}$ admits an $\Ad(H)$-invariant splitting $\mathfrak{g} =\mathfrak{m} \oplus \mathfrak{h}$. We can then identify $\mathfrak{g}/\mathfrak{h}$ with $\mathfrak{m}$. Furthermore, if we decompose the Cartan connection as $\varpi = \theta \oplus \omega$ with $\theta$ and $\omega$ taking values in respectively $\mathfrak{m}$ and $\mathfrak{h}$, then $\omega$ is a principal connection on $P$.

\subsection{Canonical sub-Riemannian Cartan connections}
Let $(M,E,g)$ be a sub-Riemannian manifold with constant symbol $\mathfrak{g}_-$. Consider $\mathfrak{g} = \mathfrak{g}_- \oplus \mathfrak{g}_0$ as in Section~\ref{sec:Isometries} with the inner product induced from $\mathfrak{g}_{-1}$. Define a representation $\Ad$ of $G_0= \Iso(\mathfrak{g}_-)$ on $\mathfrak{g}$ as the usual adjoint action on $\mathfrak{g}_0$ and defined as in \eqref{Adg-} on $\mathfrak{g}_-$. Introduce \emph{the Spencer differential} $\partial: \wedge^k \mathfrak{g}_-^* \otimes \mathfrak{g} \to \wedge^{k+1} \mathfrak{g}_-^* \otimes \mathfrak{g}$, $k=1, \dots, n$, by
\begin{align*}
(\partial \alpha)(A_0, \dots, A_k) &= \sum_{i=0}^n (-1)^i [A_i, \alpha(A_0, \dots, \hat A_i, \dots, A_k)] \\
& \qquad + \sum_{i<j} (-1)^{i+j} \alpha([A_i, A_j], A_0, \dots, \hat A_i, \dots, \hat A_j, \dots, A_k),
\end{align*}
where $A_0, \dots, A_k \in \mathfrak{g}_-$ and where the hat denotes terms that should be omitted. We get an induced inner product on $\wedge^k \mathfrak{g}_-^* \otimes \mathfrak{g}$ from the inner products of $\mathfrak{g}_-$ and $\mathfrak{g}$. More specifically, if $A_1, \dots, A_{n}$ is an orthonormal basis of $\mathfrak{g}_-$, with dual $A_1^*, \dots, A_n^*$ and $D_1, \dots, D_m$ is an orthonormal basis of $\mathfrak{g}_0$, then we define an inner product on $\wedge^k \mathfrak{g}_-^* \otimes \mathfrak{g}$ such that,
$$\left\{ \begin{array}{c} A_{i_1}^* \wedge \cdots \wedge A_{i_k}^* \otimes A_r, \\ A_{j_1}^* \wedge \cdots \wedge A_{j_k}^* \otimes D_s \end{array}  \quad :  \quad \begin{array}{c} i_1 < \dots < i_k, j_1 <\dots < j_k, \\ 1 \leq r \leq n, 1\leq s \leq m \end{array} \right\},$$
is an orthonormal basis. Write $\partial^*$ for the dual of $\partial$ with respect to our mentioned inner product.

For a connection $\varpi = \theta \oplus \omega$ on $\scrF$, we say that it is \emph{adapted} or \emph{regular} if
$$\theta(w) = f^{-1} (\pi_* w \bmod E^{-j+1}) \mod \frakg_{-j+1}, \qquad \text{for $w \in T_f \scrF$ with $\pi_* w \in E^{-j}$.}$$
We then have the following result by Morimoto \cite{Mor08}.
\begin{theorem} \label{th:UniqueCartan}
Let $(M, E,g)$ be a sub-Riemannian manifold with constant symbol~$\mathfrak{g}_-$. Let $G_0 \to \mathscr{F} \to M$ be its nonholonomic frame bundle. Then there is a unique adapted $(\mathfrak{g},\mathfrak{g}_0)$-Cartan connection $\varpi: T\mathscr{F} \to \mathfrak{g}$ such that its curvature $\kappa: \scrF \to \wedge^2 \mathfrak{g}_-^* \otimes \mathfrak{g}$ satisfies
$$\partial^* \kappa = 0.$$
\end{theorem}
We note the following important point related to isometries. Let $(M, E, g)$ and $(M', E', g')$ be two sub-Riemannian manifolds with constant symbol $\mathfrak{g}_-$. Assume that there is an isometry $\Phi: M \to M'$, i.e. a diffeomorphism such that $\Phi_*$ maps $E$ to $E'$ isometrically on each fiber. Write $\gr = \gr E$ and $\gr' = \gr E'$ and let $\mathscr{F}$ and $\mathscr{F}'$ be their respective non-holonomic frame bundles, and consider the induced map
$$\Phi_*: \mathscr{F} \to \mathscr{F}', \qquad \Phi_*: f \in \Iso(\mathfrak{g}_-, \gr_x) \mapsto (\Phi_{*,x} \circ f) \in \Iso(\mathfrak{g}_-, \gr_{\Phi(x)}'),$$
We can then verify that if $\varpi'$ is the unique Cartan connection on $\mathscr{F}'$ satisfying Theorem~\ref{th:UniqueCartan}, then $\varpi =\Phi^*\varpi'$ will also satisfy the same condition on $\mathscr{F}$ and their curvatures are related by $\kappa'(\Phi(p)) = \kappa(p)$.

\begin{remark}
In this paper, we are considering a slightly different normalization condition in \cite{Mor08}, with our convention of how to extend the metric from an inner product space to its tensors Section~\ref{sec:Free}. We make this convention to simplify our presentation. In particular, our formula for the selector $\chi_I$ in Section~\ref{sec:GradingSelector} and our result in Theorem~\ref{th:MorimotoGC} would have been more complicated without this convention.
\end{remark}
%
%
%
%

\subsection{Unique affine connection and grading}
We can rewrite Morimoto's connection using the concepts of $E$-gradings and affine connections. Let $(M,E,g)$ be a sub-Riemannian manifold with constant symbol $\mathfrak{g}_-$. Let $I$ be a chosen $E$-grading, which induces a Riemannian metric $g_I = \langle \cdot , \cdot \rangle_I$ and a selector $\chi_I$ as in Section~\ref{sec:GradingSelector}. Note that we can now consider each vector space $T_xM$ as a Carnot algebra with brackets $-\mathbb{T}$, defined as in \eqref{BBoldT}. Define $\mathfrak{s}_x = \iso(T_xM)$ as the isometry algebra at each point which together form a subbundle $\mathfrak{s}$ of $\End TM$.

With these notations established, we are ready to present a canonical choice of grading and affine connection on a sub-Riemannian manifold with constant symbol.
\begin{theorem} \label{th:MorimotoGC}
There is a unique $E$-grading $I$ and affine connection $\nabla$ on $TM$ satisfying the following.
\begin{enumerate}[\rm (i)]
\item $\nabla$ is strongly compatible with $(E,g,I)$;
\item for any $D \in \mathfrak{s}$ and any $v \in E^{-i}$, $w \in E^{-j}$ with $0 \leq j < i \leq s$,
\begin{align} \label{Rcond}
\langle R(\chi_I(v)), D \rangle_{g_I} & = \langle T(v, \cdot), D \rangle_{I} , \\ \label{Tcond}
\langle T(\chi_I(v)), w \rangle_{g_i} & =   - \langle T(v, \cdot) , \mathbb{T}(w, \cdot) \rangle_{I} .
\end{align}
\end{enumerate}
\end{theorem}
We will refer to the grading and connection satisfying the above theorem as respectively \emph{the Morimoto grading} and \emph{the Morimoto connection}.
\begin{proof}
Consider the adapted canonical Cartan connection $\varpi = \theta \oplus \omega: T\mathscr{F} \to \mathfrak{g}$ as in Theorem~\ref{th:UniqueCartan} with $\theta$ taking values in $\mathfrak{g}_-$ and $\omega$ taking values in $\mathfrak{g}_0$. Then~$\omega$ is a principal connection on $\pi$. Define $\calH = \ker \omega$ and write the corresponding horizontal lift by $h_f: T_{\pi(f)} M \to \calH_f \subseteq T_f \scrF$, that is, $h_f v \in \calH_f$ is the unique element satisfying $\pi_* h_f v = v$. Define a grading $I: TM \to \gr$ by
$$I v = f \theta(h_f v), \qquad v \in T_x M, f \in \scrF_x, x \in M.$$
This is well defined since $\calH$ is a principal Ehresmann connection. Indeed, any other element in $\scrF_x$ will be on the form $f \cdot a$, $a \in G_0$, we observe that
\begin{align*}
& (f \cdot a) (\theta( h_{f \cdot a} v)) = (f \cdot a) (\theta( h_{f} v \cdot a))  = (f \cdot a) \Ad(a^{-1}) \theta( h_{f} v) \\ 
& = (f \cdot a) a^{-1} \theta( h_{f} v) = f\theta(h_f v). \end{align*}

We now turn to the curvature $\kappa$ of $\varphi$. We first observe that for $X,Y \in \Gamma(TM)$, $\bar{D} \in \mathfrak{g}_0$, $\bar{B} \in \mathfrak{g}_-$, since $\varpi$ is adapted.
\begin{align*}
\langle \kappa(\theta(hX), \theta(hY))|_f, \bar{D} \rangle & = \langle -\omega([hX,hY])|_f, \bar{D} \rangle = \langle R(X,Y), I^{-1} f \bar{D} f^{-1} I \rangle_{g_I}, \\
\langle \kappa(\theta(hX), \theta(hY))|_f, \bar{B} \rangle & = \langle h_f X f^{-1} I(Y) - h_f Y f^{-1} I(X) - I([X,Y])  , \bar{B} \rangle \\
& \qquad + \langle f^{-1} \Lbra I(X), I(Y) \Rbra , \bar{B} \rangle \\
& = \langle T(X,Y) -  \mathbb{T}( X,Y) , I^{-1}(f \bar{B} )\rangle_{I} .
\end{align*}
We see that since $\varpi$ is adapted, $\kappa$ is only non-zero in positive degrees. Hence, we only need to show that $\kappa$ is orthogonal to the image of positive elements of $\mathfrak{g}_-^* \otimes \mathfrak{g}$ under $\partial$.

If $X \in (TM)_{-i}$, $Y \in (TM)_{-j}$ and $f \in \scrF_x$ with $i > j \geq 1$, then
\begin{align*}
& 0 =  \langle \kappa, \partial ((f^{-1}I X)^* \otimes (f^{-1} I Y)) \rangle \\
&  = \langle I^{-1} f\kappa f^{-1} I, -X^* \wedge \mathbb{T}(Y, \, \cdot \,)  + X^*( \mathbb{T}( \, \cdot \, , \, \cdot \, )) \otimes Y\rangle_I \\
& =-\tr_{TM}  \langle T (X, \times) ,  \mathbb{T}( Y, \times ) \rangle_I  - \langle T(\chi_I(X))  ,  Y \rangle_I .
\end{align*}
Hence, condition \eqref{Tcond} holds.

Next, if $X \in (TM)_{-i}$, $i \geq 1$, $D \in \mathfrak{s}_x$, $f\in \scrF_x$, then
\begin{align*}
0 = & \langle  \kappa , \partial ((f^{-1}IX)^* \otimes f^{-1}ID I^{-1} f) \rangle  = \langle I^{-1} f \kappa f^{-1} I , X^* \wedge D  + X^* (\mathbb{T}(\cdot, \cdot)) \otimes D \rangle_I \\
& = \tr_{TM} \langle T(X, \times) , D \times  \rangle_I -  \langle R( \chi_I( X), D \rangle_I.
\end{align*}
We have therefore proved \eqref{Rcond} and the result follows.
\end{proof}

In the following sections, we will consider examples showing how the Morimoto grading and connection can be computed.

\subsection{Strongly compatible connections and flatness}
We emphasize the following important consequence of our previous proof.
\begin{corollary}[Sub-Riemannian flatness theorem]
Let $(M, E, g)$ be a sub-Riemannian manifold with constant symbol $\mathfrak{g}_-$. Let $(G_-, E',g')$ be the Carnot group corresponding to the symbol $\mathfrak{g}_-$ with its sub-Riemannian structure.
\begin{enumerate}[\rm (a)]
\item There exist $\nabla$ is a strongly compatible connection with $(E,g,I)$ for some $E$-grading $I$ satisfying
\begin{equation} \label{Flat} T(X, Y) = \mathbb{T}(X,Y), \qquad R(X, Y) = 0,\end{equation}
for any $X, Y \in \Gamma(TM)$, then $(M,E,g)$ is locally isometric to $(G_-, E',g')$. Furthermore, $I$ and $\nabla$ are then respectively the Morimoto grading and connection of $(M,E,g)$.
\item Conversely, if $(M, E, g)$ is locally isometric to $(G_-, E',g')$, and if $I$ and $\nabla$ are respectively its Morimoto grading and connection, then they satisfy \eqref{Flat}.
\end{enumerate}
\end{corollary}

\begin{proof}
Reversing the steps in the proof of Theorem~\ref{th:MorimotoGC}, we see that any grading and strongly compatible connection can be seen as a Cartan connection $\varpi$ on the nonholonomic frame bundle $\mathscr{F}$, with the conditions \eqref{Flat} being equivalent to the curvature vanishing. Connections with curvature $\kappa \equiv 0$ will in particular satisfy Theorem~\ref{th:UniqueCartan}. Let $G$ be the simply connected Lie group corresponding to $G$. If the curvature vanishes, then by \cite[Chapter 3, Theorem 6.1]{Sha97}, for every $f \in \mathscr{F}$, there is a neighborhood $U$ and a map $\Psi:U \to G$ such that
$$\varpi(w) = \Phi(f)^{-1} \cdot \Psi_* w, \qquad w \in T_f \mathscr{F}, f \in U.$$
Following similar steps as in \cite[Chapter 5.5]{Sha97}, we use invariance under $G_0$ to reduce to a map $\bar{\Phi} : \pi(U) \to G_1$ mapping $E|_U$ into $E'$ isometrically on each fiber. This proves (a). The converse follows from observing that the Morimoto grading on a Carnot group is given by the left translation of the stratification and the Morimoto connection is the connection determined by making all left invariant vector fields parallel. 
\end{proof}

%% file: Contact.tex
\section{Sub-Riemannian manifolds with contact horizontal bundles} \label{sec:Contact}
\subsection{The Heisenberg algebra} \label{sec:Heisenberg}
The $n$-th Heisenberg algebra is the step~$2$ nilpotent algebra $\mathfrak{h}_n = \mathfrak{g}_- = \mathfrak{g}_{-2} \oplus \mathfrak{g}_{-1}$ where
$$\mathfrak{g}_{-2} = \spn \{ C\}, \qquad \mathfrak{g}_{-1} = \spn \{ A_1, \dots, A_n, B_1, \dots, B_n \},$$
and with the only non-zero brackets being $[A_j, B_j] = C$ for all $j =1, 2 \dots, n$.
We consider Carnot algebras structures on $\mathfrak{h}_n$. For any vector $\lambda = (\lambda_1, \dots, \lambda_n) \in \mathbb{R}^n$ such that
\begin{equation}
\label{Lambda} 1 = \lambda_1 \leq \lambda_2 \leq \cdots \leq \lambda_n,
\end{equation}
we define an inner product $\mathfrak{g}_{-1}$ by
$$\langle A_j , A_j \rangle_{\mathfrak{g}_-} = \langle B_j , B_j \rangle_{\mathfrak{g}_-} =\lambda_j^2, \qquad \textstyle \langle A_i, B_j \rangle_{\mathfrak{g}_-} = 0, \quad  i \neq j.$$

It is simple to verify that any Carnot algebra structure on $\mathfrak{h}_n$ is isometric to exactly one choice of vector $\lambda$ as in \eqref{Lambda}. We write this Carnot algebra as $\mathfrak{h}_n(\lambda)$.
Then the isometry algebra $\mathfrak{g}_0 = \iso(\mathfrak{h}_n(\lambda))$ is given by
$$\mathfrak{g}_0 = \spn \{ \bar{D}_{ij} \, : \, i < j , \lambda_i = \lambda_j\} \cup \{ \bar{Q}_{ij} \, : \, i \leq j , \lambda_i = \lambda_j\}.$$
with
\begin{align*}
\bar{D}_{ij}( A_k) & = \delta_{ki} A_j - \delta_{kj} A_i, &  \bar{Q}_{ij}(A_k) & =  \frac{1}{2} \delta_{ki} B_j + \frac{1}{2} \delta_{kj} B_i, \\
\bar{D}_{ij}( B_k) & = \delta_{ki} B_j - \delta_{kj} B_i, &  \bar{Q}_{ij}(B_k) & =  -\frac{1}{2} \delta_{ki} A_j - \frac{1}{2} \delta_{kj} A_i,\\
\bar{D}_{ij} C &= 0, &  \bar{Q}_{ij} C &= 0.\end{align*}

\subsection{Contact manifolds with constant symbols} \label{sec:ContactStructures}
A subbundle $E$ on $M$ is called a contact distribution if it has codimension $1$ and if the bilinear map
$$\wedge^2 E \to \gr_{-2} = TM/E^{-2}, \qquad X \wedge Y \mapsto \Lbra X, Y \Rbra = [X,Y] \bmod E^{-2},$$
is non-degenerate. It follows that $E$ is necessarily bracket generating, equiregular and of even rank. We write this rank as $2n$, meaning that the manifold $M$ has dimension $2n+1$. If $(M, E, g)$ is a sub-Riemannian manifold $E$ a contact distribution, then $\gr_x$ is isometric to the Heisenberg algebra with a given inner product, but this inner product will depend on $x \in M$ in general.

Working locally, we may assume that $\Ann(E)$ is a trivial line bundle. Choose any orientation. Let $\theta$ be a positively oriented non-vanishing section of $\Ann(E)$. We introduce a corresponding map $J^\theta: E \to E$ defined such that
$$d\theta(u,v) = \langle u, J^\theta v \rangle_g.$$
Observe that $J^\theta$ is a skew-symmetric map with a trivial kernel and hence $-(J^\theta)^2$ is positive definite. We normalize $\theta$ by requiring it to be the unique one-form such that the maximal eigenvalue of $-(J^\theta)^2$ at every point is 1.

Notice that every eigenvalue of $-(J^\theta)^2$ has at least multiplicity $2$, since if $u$ is an eigenvector, then $J^\theta u$ will also be an eigenvector with the same eigenvalue. Hence, we can describe all of the eigenvalues of $J^\theta$ at every point by a function
$$M \mapsto \mathbb{R}^n, \qquad x \mapsto \lambda_x = (\lambda_{x,1}, \dots, \lambda_{x,n}), \qquad 1 = \lambda_{x,1} \leq \cdots \leq \lambda_{x,n},$$
such that the eigenvalues, with multiplicities are
$$1 = 1 \geq \frac{1}{\lambda_{x,2}^2} = \frac{1}{\lambda_{x,2}^2} \geq \cdots \geq \frac{1}{\lambda_{x,n}^2} = \frac{1}{\lambda_{x,n}^2}.$$
It is then simple to verify that the symbol of $(M, E, g)$ at $x$ is $\mathfrak{h}_n(\lambda_x)$, and so the manifold has constant symbol if and only if $\lambda_x = \lambda$ is independent of $x\in M$.

We will continue the discussion assuming that $(M,E,g)$ has constant symbol. Let $1 = \lambda[1] <  \lambda[2]< \cdots < \lambda[k]$ be the values of $\lambda =(\lambda_1, \dots, \lambda_n)$ written without repetitions, and write the corresponding eigenspace decomposition of $E$ by
\begin{equation} \label{ELambda} E = E[1] \oplus \cdots \oplus E[k],\end{equation}
where $E[j]$ is the eigenspace of $\lambda[j]^{-2}$. Define $\pr[j] : E \to E[j]$ as the corresponding projection. The decomposition in \eqref{ELambda} is an orthogonal since $-(J^\theta)^2$ is symmetric. Define a vector bundle map $\Lambda : E\to E$ by
$$\Lambda |_{E[j]} = \lambda[j] \id_{E[j]}.$$
Write $J^\theta = \Lambda^{-1} J$ and notice that we now have $J^2 = -\id_E$.

We define \emph{the Reeb vector field} $Z^0$ such that $\theta(Z^0) =1$ and $d\theta(Z^0, \, \cdot \,) = 0$ and denote its span by $V^0 = \spn \{Z^0\}$. We then observe that any $E$-grading $I$ is uniquely determined by the vector field $W$ with values in $E$ and defined by
$$W = J \pr_{-1} Z^0.$$
Write $\ker \pr_{-1} = V^W = \spn \{ Z^W\}$ where $Z^W = Z^0 - JW$. We observe that for any $X \in \Gamma(E)$,
$$d\theta(Z^W, X) = - \theta([Z^W, X])  = - \langle JW, J^\theta X \rangle =- \langle \Lambda^{-1} W, X \rangle.$$
Define $g_I$ as the taming Riemannian metric of $g$ such that that $Z^W$ is a unit vector orthogonal to~$E$. We also extend the definition of $\Lambda$, $\Lambda^{-1}$, $J$ and the maps $\pr[j]$ to the whole tangent bundle by requiring them to vanish on $Z^W$. Write $\pr_0[j]$ for the projection to $E[j]$ extended by $\pr_0[j]Z^0 =0$.

We observe that the decomposition \eqref{ELambda}, the map $\Lambda$, the vector field $W$, the subbundle $V^W$ and the metric $g_I$ does not depend on the choice of orientation $\theta$, and can hence be defined globally. Since we are in step~$2$, we have a canonical selector $\chi: \wedge^2 TM \to TM$. Locally, this is given by
$$\chi(Z^W) = \chi(Z^0) =:  \chi_Z =  \frac{2}{\tr \Lambda^{-2}} \sum_{j=1}^n \Lambda^{-1} X_j \wedge JX_j,$$
where $X_1, \dots, X_n$ is a choice of local unit vector fields such that
\begin{equation} \label{Jbasis} \{ X_1, \dots, X_n, J X_1, \dots,  J X_n \} \qquad \text{is a local orthonormal basis.}
\end{equation}

\subsection{Connections on contact manifolds} \label{sec:ContactConnection}
For the case when $\Lambda$ is a projection to $E$, i.e., the case when $1$ is the only eigenvalue of $J^\theta$, there exists several known canonical choices of connections defined by Tanaka and Webster for the integrable case \cite{Tan76,Web78} and by Tanno \cite{Tan89} when $J = J^\theta$ is non-integrable almost complex structure on $E$. We will introduce the following generalization.

Let $I$ be any $E$-grading with $\ker I_{-1} = V^W$. Introduce tensor $\tau(X)Y = \tau_X Y$ by
$$\langle \tau_X Y_1, Y_2 \rangle = \frac{1}{2} \sum_{j=1}^k \left(\calL_{X- [X]_j} g_I)( [Y_1]_j, [Y_2]_j\right).$$
where $[X]_j = \pr[j] X$ with similar notation for $Y_1$ and $Y_2$.
Let $\nabla^{I}$ be the Levi-Civita connection of $g_I$. Define first a connection $\nabla' = \nabla^{\prime,W}$ such that it is compatible with $g_I$ and such that each of the subbundles in the decomposition $TM = E[1] \oplus \cdots \oplus E[k] \oplus V^W$ are parallel. It follows that if $Z^W$ is defined with respect to any local orientation, then we must have $\nabla' Z^W =0$. For a horizontal vector field $X \in \Gamma(E)$, we define for $Y \in \Gamma(TM)$,
$$\nabla_Y X = \sum_{j=1}^k \pr[j] \nabla_{[Y]_j}^I [X]_j + \sum_{j=1}^k \pr[j] \big[ Y -[Y]_j,   [X]_j\Big] + \tau_{Y} X.$$ 
We observe that for any $X_1, X_2 \in \Gamma(E[j])$ and any $Y \in \Gamma(TM)$, we have
$$\langle T'(Y,X_1), X_2 \rangle = \langle \tau_Y X_1, X_2 \rangle, \qquad T'(X_1, X_2) = \mathbb{T}(X_1, X_2),$$
where $T'$ is the torsion $\nabla'$.

We next define a connection $\nabla'' = \nabla^{\prime\prime,W}$ by
$$\nabla''_{Y_1} Y_1 = \nabla_{Y_1}' Y_2  + \frac{1}{2}  (\nabla'_{Y_1} J)JY_2, \qquad Y_1, Y_2 \in \Gamma(TM).$$
We observe the following.
\begin{lemma} \label{lemma:Tprimeprime}
\begin{enumerate}[\rm (a)]
\item The connection $\nabla'$ is compatible with $(E,g,I)$ and $\nabla''$ is strongly compatible with $(E,g,I)$.
\item For any $X,X_1 \in \Gamma(E)$,
$$\langle (\nabla_{X_1}' J) X, X_1 \rangle_{I} + \langle (\nabla_{JX_1}' J)X, JX_1 \rangle_{I} = 0.$$
In particular, for any $j =1 , \dots, k$, 
\begin{align} \label{TraceJ} \tr_{E[j]} \langle (\nabla_{\times}' J) X, \times \rangle_{I} = - \tr_{E[j]} \langle (\nabla_{\times}' J) \times, X \rangle_{I} =0.\end{align}
\item If $T''$ is the torsion of $\nabla''$ and $Y \in \Gamma(TM)$, then for any $D \in \mathfrak{s}$,
$$\langle T_Y'', D \rangle_{I}= 0 .$$
\end{enumerate}
\end{lemma}

\begin{proof}
We will be working locally thought the proof. We will always let $Z = Z^W = Z^0 - JW$ denote the modification of the Reeb vector field with respect to a chosen orientation on $\Ann(E)$. To simplify notation in the proof, we will simply write $\langle \cdot , \cdot \rangle_{I}$ as just $\langle \cdot , \cdot \rangle$.
\begin{enumerate}[\rm (a)]
\item We first note that $\nabla^\prime$ is compatible with $(E, g,I)$ by definition. Since for any $X_1,X_2 \in \Gamma(E)$, $Y \in \Gamma(TM)$, we have
$$\langle X_1,  (\nabla'_Y J) JX_2 \rangle_g = \langle J (\nabla'_Y J) X_1, X_2 \rangle = - \langle (\nabla'_Y J) J X_1, X_2 \rangle,$$
the connection $\nabla''$ is also compatible with $(E,g)$. Here we have used that $J^2 = -\pr_E$ is $\nabla'$-parallel.
Finally, for any $Y_1, Y_2, Y_3 \in \Gamma(TM)$,
\begin{align*}
& (\nabla_{Y_3}'' \mathbb{T})(Y_1, Y_2) = \langle Y_1, \Lambda^{-1} (\nabla_{Y_3}'' J) Y_2 \rangle Z \\
&=  \langle Y_1, \Lambda^{-1} (\nabla_{Y_3}' J) Y_2 \rangle Z  \\
& \qquad +\frac{1}{2} \langle Y_1, \Lambda^{-1} (\nabla_{Y_3}' J)J^2 Y_2 \rangle Z - \frac{1}{2} \langle Y_1, \Lambda^{-1}J (\nabla_{Y_3}' J)J Y_2 \rangle Z  = 0.
\end{align*}

\item For $X_1, X_2 \in \Gamma(E[j])$, $X \in \Gamma(E[i])$, we use the first Bianchi identity for connections with torsion, with $\circlearrowright$ denoting the cyclic sum
\begin{align*}
&\circlearrowright \langle   \Lambda^{-1} (\nabla_X' J) X_1, X_2 \rangle = - \langle \circlearrowright (\nabla'_X T')(X_1, X_2), Z \rangle \\
& = - \langle \circlearrowright R'(X ,X_1) X_2, Z \rangle  + \langle \circlearrowright T'(T'(X,X_1), X_2), Z \rangle  = \circlearrowright \langle T'(X,X_1), \Lambda^{-1} JX_2  \rangle.
\end{align*}
If $i = j$, we have
\begin{align} \label{Bianchi1}
\circlearrowright \langle  (\nabla_X' J) X_1, X_2 \rangle  & =0.
\end{align}
If we insert $X = JX_1$, we obtain
\begin{align} \label{Bianchi2}
\langle (\nabla_{JX_1}' J)X_1, X_2 \rangle &= - \langle (\nabla_{X_2}' J) J X_1, X_1 \rangle - \langle (\nabla_{X_1}' J) X_2, J X_1 \rangle \\ \nonumber
&  =  - \langle J(\nabla'_{X_1} J) X_1, X_2 \rangle.
\end{align}
It then follows that
\begin{align*}
& \langle (\nabla_{JX_1}' J) X_2, JX_1 \rangle = \langle X_2, J(\nabla_{JX_1}' J) X_1 \rangle = \langle X_2, (\nabla_{X_1}' J) X_1 \rangle = - \langle (\nabla_{X_1}' J) X_2, X_1 \rangle.
\end{align*}
\item Let $X_1, \dots, X_n \in \Gamma(E)$ be an orthonormal set of vector fields satisfying \eqref{Jbasis} and such that each vector field is contained in exactly one subbundle $E[p]$, $p =1, \dots, k$. If $X_i, X_j$ both take values in $E[p]$, we define
\begin{align*}
D_{ij} Y & = \langle Y, X_i \rangle X_j - \langle Y, X_j \rangle X_i + \langle Y, JX_i \rangle JX_j - \langle Y, JX_j \rangle JX_i, \\
2Q_{ij} Y & = \langle Y, X_i \rangle JX_j + \langle Y, X_j \rangle JX_i - \langle Y, JX_i \rangle X_j - \langle Y, JX_j \rangle X_i,
\end{align*}
The elements of $\mathfrak{s}$ are spanned by $D_{ij}$ and $Q_{ij}$ from Section~\ref{sec:Heisenberg}.
Observe that
\begin{align*}
& \langle T_Z'' , Q_{ij} \rangle = \langle T_Z', Q_{ij} \rangle + \frac{1}{2} \langle (\nabla_Z' J)J, Q_{ij}\rangle  = 0 \end{align*}
and
\begin{align} \label{TZeta}
& \langle T_Z'' , D_{ij} \rangle  = \langle \tau_Z, D_{ij} \rangle  + \frac{1}{2} \langle (\nabla_Z' J)J, D_{ij} \rangle \\ \nonumber
& =\frac{1}{2} \langle (\nabla_Z' J)J X_i, X_j  \rangle - \frac{1}{2} \langle (\nabla_Z' J)JX_j, X_i \rangle \\ \nonumber
& \qquad - \frac{1}{2} \langle (\nabla_Z' J) X_i, JX_j \rangle + \frac{1}{2} \langle (\nabla_Z' J)X_j, JX_i \rangle = 0. \end{align}
For $X \in \Gamma(E[q])$ and with $p \neq q$, then by a similar calculation as in \eqref{TZeta},
$$\langle T''_X,  D_{ij} \rangle = 0.$$
If $X$ is also in $E[p]$, then
\begin{align*}
\langle T''_X,  D_{ij} \rangle & = \frac{1}{2} \langle (\nabla'_X J)J, D_{ij} \rangle - \frac{1}{2} \langle (\nabla'_{\cdot}J )J X, D_{ij}\rangle =- \frac{1}{2} \langle (\nabla'_{\cdot}J )J X, D_{ij}\rangle. 
\end{align*}
which means that
\begin{align*}
& -2 \langle T''_X,  D_{ij} \rangle= \langle (\nabla'_{\cdot}J )JX, D_{ij}\rangle.  \\
& = \langle (\nabla'_{X_i} J) JX , X_j \rangle - \langle (\nabla'_{X_j} J) JX, X_i \rangle  + \langle (\nabla'_{JX_i} J) J X, JX_j \rangle - \langle (\nabla'_{JX_j} J) JX, JX_i \rangle \\
& =  \langle (\nabla'_{X_i} J) JX_j , X \rangle - \langle (\nabla'_{X_j} J) X, JX_i \rangle  - \langle (\nabla'_{JX_i} J) X_j , X \rangle + \langle (\nabla'_{JX_j} J) X, X_i \rangle \\
& =  \circlearrowright   \langle (\nabla'_{X_i} J) JX_j , X \rangle -   \langle (\nabla'_{X} J) X_i , JX_j \rangle  -  \circlearrowright \langle (\nabla'_{X_j} J) X, JX_i \rangle  + \langle (\nabla'_{X} J) JX_i, X_j \rangle  \\
& \stackrel{\eqref{Bianchi1}}{=} 0.
\end{align*}
Finally,
\begin{align*}
&2 \langle T''_X , Q_{ij} \rangle = 2 \langle \tau_X , Q_{ij} \rangle + \langle (\nabla'_X J)J, Q_{ij} \rangle - \langle  (\nabla'_{\cdot} J) JX, Q_{ij} \rangle  \\
& = \frac{1}{2} \langle (\nabla'_X J)JX_i, JX_j \rangle + \frac{1}{2} \langle (\nabla'_X J)JX_j, JX_i \rangle \\
& \qquad  - \frac{1}{2} \langle (\nabla'_X J)J^2 X_i, X_j \rangle - \frac{1}{2} \langle (\nabla'_X J)J^2X_j, X_i \rangle \\
& \qquad  - \frac{1}{2} \langle  (\nabla'_{X_i} J) JX, JX_j \rangle - \frac{1}{2} \langle  (\nabla'_{X_j} J) JX, JX_i \rangle \\
& \qquad + \frac{1}{2} \langle  (\nabla'_{JX_i} J) JX, X_j \rangle + \frac{1}{2} \langle  (\nabla'_{JX_j} J) JX, X_i \rangle \\
& =   \frac{1}{2} \langle  (\nabla'_{X_i} J) X, X_j \rangle + \frac{1}{2} \langle  (\nabla'_{X_j} J) X, X_i \rangle \\
& \qquad + \frac{1}{2} \langle  (\nabla'_{JX_i} J) X, JX_j \rangle + \frac{1}{2} \langle  (\nabla'_{JX_j} J) X, JX_i \rangle
\end{align*}
by (b). \qedhere
\end{enumerate}
\end{proof}

\subsection{The Morimoto grading and connection for the contact case}
We will now present the Morimoto connection for a sub-Riemannian manifold $(M, E, g)$ of constant symbol $\mathfrak{h}_n(\lambda)$.

Working locally, let $J$ be defined relative to a choice of orientation on $\Ann(E)$. For any $i,j = 1, \dots, k$ with $i\neq j$, define vector fields,
$$\Upsilon_{ii} =0, \qquad \Upsilon_{ij} = \frac{1}{2}  \tr_E \pr[i] \Big[ \pr[j] \times , \pr[j] J\times \Big].$$
We observe that these vector fields are independent of orientation and are hence well defined globally. Recall the definition of the subbundles $V^W$ from Section~\ref{sec:ContactStructures} and the connection $\nabla^{\prime\prime,W}$ from Section~\ref{sec:ContactConnection}.

\begin{theorem} \label{th:MorContact}
The Morimoto grading $I$ is the unique grading with
$$\ker I_{-1} = V^W, \qquad W := - \frac{2}{\tr \Lambda^{-2}} \sum_{i,j=1}^k \frac{\lambda[i]^2}{\lambda[j]} J\Upsilon_{ij} .$$
Furthermore if we consider the connection $\nabla'' = \nabla^{\prime\prime,W}$ with curvature $R''$, then the Morimoto connection is given by,
$$\nabla_{Y_1} Y_2 = \nabla_{Y_1}'' Y_2+ \frac{1}{2} R''(\chi(Y_1) )Y_2, \qquad Y_1, Y_2 \in \Gamma(TM).$$

\end{theorem}

\begin{proof}
We will work locally, considering $Z^0$ as the Reeb vector field with corresponding map $J$. We write the vector field $Z = Z^W = Z^0 - J W$.

We already know that $\nabla''$ is strongly compatible with $(E,g,I)$. By compatibility with $g_I$ and strong compatibility, it follows that $R''(\chi(Y_1))$ is $g_I$-skew-symmetric map that preserves the grading and commutes with $J$. This gives us that $\nabla$ is also strongly compatible with $(E, g,I)$.

We first see that it satisfies \eqref{Tcond}, which can be written as $\langle T(\chi_Z), X \rangle =  \lambda[p]^{-1}  \langle T(Z, JX) ,Z \rangle$ for any $X \in \Gamma(E[p])$. Observe that
$$-T''(\chi_Z) \stackrel{\eqref{TraceJ}}{=} - T'(\chi_Z) =Z + \frac{2}{\tr \Lambda^{-2}} \sum_{i,j=1}^k \lambda[j]^{-1} \Upsilon_{ij}.$$
Hence, for any $X \in \Gamma(E[i])$,
\begin{align*}
& \langle T(\chi_Z) , X \rangle_{I} = \langle T''(\chi_Z)  , X \rangle_{I}  = - \frac{2}{\tr \Lambda^{-2}} \sum_{j=1}^k \lambda[j]^{-1}  \langle \Upsilon_{ij}, X \rangle_{I}  \\
&  =\lambda[i]^{-1} \langle T( Z, JX), Z \rangle_{I}  = \lambda[i]^{-1} d\theta(Z, JX) = - \lambda[i]^{-2} \langle W, JX \rangle_I = \lambda[i]^{-2} \langle JW, X \rangle_I.
\end{align*}

To see that \eqref{Rcond} is also satisfied, we see that for any $X \in \Gamma(E)$,
$$ 0 = R(\chi(X)) = \pr_{\mathfrak{s}} T_X'' = 0,$$
by Lemma~\ref{lemma:Tprimeprime}. Finally, by Lemma~\ref{lemma:Tprimeprime}~(d),
\begin{align*}
& R(\chi_Z) = R''(\chi_Z ) -   \frac{1}{2} R''(\chi_Z ) = \frac{1}{2} R''(\chi_Z )\\
& = \pr_{\mathfrak{s}} T_Z = \pr_{\mathfrak{s}} T_Z'' + \frac{1}{2} R''(\chi_Z )= \frac{1}{2} R''(\chi_Z ).\qedhere
\end{align*}
\end{proof}

\begin{remark}[Explicit description of $W$]
The definition of $W$ in Theorem~\ref{th:MorContact} is implicit, as $W$ is also used by the projections in the definitions of $\Upsilon_{ij}$. For explicit description, note that if $i \neq j$ and we define
$$\Upsilon_{ii} =0, \qquad \Upsilon_{ij}^0 = \frac{1}{2}  \tr_E \pr_0[i] \Big[ \pr_0[j] \times , \pr_0[j] J\times \Big].$$
Then for $i \neq j$ and if $\rank E[j] = 2n[j]$, then
$$\Upsilon_{ij} = \Upsilon_{ij}^0 + n[j]JW.$$
We can then write
\begin{align*}
W & = - \frac{2}{\tr \Lambda^{-2}} \sum_{i,j=1}^k \frac{\lambda[i]^2}{\lambda[j]} J\Upsilon_{ij}^0 + \frac{2}{\tr \Lambda^{-2}} \sum_{i \neq j} \frac{\lambda[i]^2n[j]}{\lambda[j]}  W \\
& =  - \frac{2}{\tr \Lambda^{-2}} \sum_{i,j=1}^k \frac{\lambda[i]^2}{\lambda[j]} J\Upsilon_{ij} ^0+ 2 \frac{\tr \Lambda^{-1}}{\tr \Lambda^{-2}} \sum_{i =1}^k \lambda[i]^2 W  - 2\frac{\tr \Lambda}{\tr \Lambda^{-2}} W
\end{align*}
which we can solve for $W$.
\end{remark}

\begin{remark}[Alternative choice of connection]
When working with the equivalence problem on sub-Riemannian contact manifold, it may be simpler to use grading $I$ with $\ker I^{-1} = V^0 = \ker Z^0$ just given by the Reeb vector field and with the connection $\nabla''$ as this will always be strongly compatible.
\end{remark}

%% file: Cartan235.tex
\section{The Cartan algebra and $(2,3,5)$-sub-Riemannian manifolds} \label{sec:235}
\subsection{On the Carnot nilpotent algebra} \label{sec:CarnotNil}
We want to consider sub-Riemannian manifolds with growth vector $\mathfrak{G} = (2,3,5)$. These are all of constant symbol, since, up to isometry, there only exists one Carnot algebra with this growth vector; \emph{the Cartan nilpotent algebra}. It can be written as $\mathfrak{g} = \mathfrak{g}_{-3} \oplus \mathfrak{g}_{-2} \oplus \mathfrak{g}_{-1} = \spn \{ C_1, C_1\} \oplus \spn \{ B\} \oplus \spn \{ A_1, A_2 \}$ with $A_1, A_2$ being an orthonormal basis of $\mathfrak{g}_{-1}$ and with the only non-zero brackets given by
$$[A_1, A_2] = B, \qquad [A_j, B]= C_j.$$
The corresponding isometry algebra is given by $\mathfrak{g}_0 = \spn \{ \bar{J}\}$ where
$$\bar{J}:  A_1 \mapsto A_2, \quad A_2 \mapsto - A_1, \quad B \mapsto 0, \quad C_1 \mapsto C_2, \quad C_2 \mapsto - C_1.$$

\subsection{Structure on $(2,3,5)$ manifolds} \label{sec:Structure235}
Let $(M,E,g)$ be a sub-Riemannian manifold with growth vector $(2,3,5)$. We will begin by working locally, so we may assume that $E$ is trivializable and equipped with an orientation. Define an endomorphism $J: E \to E$ such that for any unit vector $u \in E_x$, $x\in M$, we have that $u, J u$ is a positively oriented orthonormal basis of $E_x$. Define an $E$-valued one-form $\varphi: TM \to E$ such that $\ker \varphi = E^{-2}$ and for $X \in \Gamma(E)$,
\begin{equation} \label{phi235} \varphi([X,[X, JX]])) = \| X\|^2_g JX.\end{equation}
In this definition, we used that the map $-JX \stackrel{\cong}{\mapsto} \| X\|^{-2}_g [X,[X,J X]] \bmod E^{-2}$ is an invertible linear map from $E$ to $TM/E^{-2}$ and hence the map $\varphi$ is well defined by
$$
\xymatrix{
TM \ar[r]^\varphi \ar[rd]& E \ar[d]^\cong\\
&  TM/E^{-2} 
} .
$$
We observe that $\varphi$ does not change if we change the orientation, and so it is globally defined.

Introduce the following two-form on $M$,
$$\Psi(v,w) := \langle J \varphi(v), \varphi(w) \rangle_g.$$
Let $I': TM \to \gr E$ be an $E$-grading, also written as $TM = (TM)_{-3}' \oplus (TM)_{-2}' \oplus (TM)_{-1}'$. Since we are working locally, we may again assume that $(TM)'_{-2}$ is a trivial line bundle. Let $\theta$ be the unique one-form with $\ker \theta = (TM)_{-3}' \oplus (TM)_{-1}'$ and with normalization condition
$$d\theta(u, v) = \langle u, Jv \rangle_g, \qquad u,v \in E.$$
We will use the forms $\Psi$ and $\theta$ to determine a grading. We emphasize that $\Psi$ is independent of $I'$ while $\theta$ is not.
\begin{lemma} \label{lemma:Grading235}
There is a unique grading $I'$ such that
\begin{enumerate}[\rm (i)]
\item for any $u \in E = (TM)_{-1}'$ and $w_1, w_2 \in (TM)_{-3}'$,
$$d\Psi(u, w_1, w_2) = 0.$$
\item for any $u\in E$, $w \in (TM)_{-2}' \oplus (TM)_{-3}'$
$$d\theta(u, w) = 0.$$
\end{enumerate}
\end{lemma}
Changing the orientation will only change the sign of $\Psi$ and $\theta$, so the grading $I'$ can hence be defined globally. To our knowledge, this is the first time this canonical grading has been observed on a $(2,3,5)$ manifold with a metric on the horizontal bundle, though some similar ideas have been considered in \cite{Sag08}.
\begin{proof}
Working locally, let $X_1$ and $X_2$ be a positively oriented orthonormal basis of $E$. Define one-forms $\alpha_1, \alpha_2, \alpha_3, \alpha_4, \alpha_5$ as the dual basis of
$$X_1,X_2, [X_1,X_2], [X_1,[X_1,X_2]], [X_2,[X_1,X_2]].$$
Then $\Psi = \alpha_4 \wedge \alpha_5$.
We note that $\Psi$ and $\spn\{ \alpha_4, \alpha_5\} = \Ann E^{-2}$ is independent of choice of basis. Observe also that for some functions $a_1, a_2, a_3 \in C^\infty(M)$,
\begin{align*}
d\Psi & = \alpha_3 \wedge \alpha_1 \wedge \alpha_5 - \alpha_3 \wedge \alpha_2 \wedge \alpha_4 + a_1 \alpha_1 \wedge \Psi + a_2 \alpha_2 \wedge \Psi  + a_3 \alpha_3 \wedge \Psi \\
& = \theta \wedge ( \alpha_1 \wedge \alpha_5 -  \alpha_4 \wedge \alpha_2   + a_3 \Psi)
\end{align*}
where
$$\theta = \alpha_3 - a_1 \alpha_4 -a_2 \alpha_5.$$
Hence, in order for (i) to be satisfied, we must have $(TM)_{-1}' \oplus (TM)_{-3}' = \ker \theta$ with $\theta$ given as above.
Furthermore, since $d\theta(X_2,X_1) = 1$, we can define
$$\beta_1 = d\theta(X_2, \, \cdot ), \qquad \beta_2 = - d\theta(X_1, \, \cdot \,).$$
Then $\beta_1$ and $\beta_2$ will depend on the choice of $X_1,X_2$, but $\spn\{ \beta_1, \beta_2\}$ will be independent of this basis. Define
$$(TM)_{-2}' \oplus (TM)_{-3}' = \ker \beta_1 \cap \ker \beta_2,$$
which satisfies (ii).
\end{proof}


\begin{proof}[Proof of Theorem~\ref{th:Flat235}]
We observe from the proof of Lemma~\ref{lemma:Grading235}. that the grading $I'$ coincides with the grading Theorem~\ref{th:Flat235}. The connection defined in Theorem~\ref{th:Flat235} is strongly compatible, so if it is flat and has $T = \mathbb{T}$, then we have local isometries to the Cartan nilpotent group. Conversely, for the Cartan nilpotent group, we observe that the grading defined by left translation of $\mathfrak{g}_{-3} \oplus \mathfrak{g}_{-2} \oplus \mathfrak{g}_{-1}$ satisfy Lemma~\ref{lemma:Grading235}. Furthermore, we can verify that the connection in Theorem~\ref{th:Flat235} is the left invariant connection. The result follows.
\end{proof}

\subsection{The Morimoto grading and connection}
Let $I'$ be as in Lemma~\ref{lemma:Grading235} and define endomorphism $\ell': TM \to TM$ by the rules
\begin{equation} \label{ellmap}
\ell' X = \pr_{-3}' \ell' \pr_{-1}' X, \qquad \varphi(\ell' X) = X.
\end{equation}
We want to use this grading to parametrize all other gradings.

We first describe everything locally. Let $Z'$ be a local choice of basis $(TM)_{-2}$ and let $J$ be the corresponding orientation of $(TM)_{-1} = E$ as described in Section~\ref{sec:Structure235}. Let $I$ be any other grading and define $\ell$ analogous to \eqref{ellmap} with respect to this grading. We observe that with respect to our conventions on the metric $g_I$,
$$\langle \ell X, \ell Y \rangle_{I}=  \langle X, Y \rangle_{I}, \qquad X, Y \in \Gamma(E).$$
Define vector fields $W_1, W_2 \in \Gamma(E)$ and endomorphism $\mathscr{A} \in \Gamma(\End E)$ by
$$I_{-1}Z' =- JW_1, \qquad \ell X = \ell' X + \langle JW_2, X \rangle Z' + \mathscr{A} X.$$
All $E$-gradings are uniquely parametrized in this way and we write the corresponding grading as $I^{W_1, W_2, \mathscr{A}}$. Let $Z = Z' + JW_1$ be the unit vector field spanning $(TM)_{-2}$. We note also that $\mathfrak{s} = \spn \{ D \}$ with $D$ given by
$$D X = JX, \qquad DZ = 0, \qquad D\ell X = \ell JX.$$

Next, we turn to connections. We observe that we are in the case of Remark~\ref{re:Maxg0}, so any compatible connection on $E$ can be extended to a strongly compatible connection $\nabla$. This will happen according to the following rules
$$\nabla_Y Z = 0, \qquad \nabla_Y \ell X = \ell \nabla_Y X,$$
for any $Y \in \Gamma(TM)$, $X \in \Gamma(E)$. For any vector field $Y$, introduce the endomorphism $\tau_Y$ of $E$ defined by
$$\langle \tau_{Y} X_1, X_2 \rangle = \frac{1}{2} (\calL_{(\pr_{-2} + \pr_{-3}) Y} g_I)(X_1, X_2),$$
which is tensorial in $Y$. Define a fixed strongly compatible connection $\nabla^0 = \nabla^{0,W_1, W_2, \calA}$ according to the formula
$$\nabla_Y^0 X = \pr_{-1} \nabla^I_{\pr_{-1} Y} X + \pr_{-1}[(\pr_{-2} + \pr_{-3}) Y, X] + \tau_Y X, \quad X\in \Gamma(E), Y \in \Gamma(TM),$$
where $\nabla^I$ is the Levi-Civita connection of $g_I$. For any $X_1, X_2 \in \Gamma(E)$, $Y \in \Gamma(TM)$, if $T^0$ is the torsion of $\nabla^0$, then
$$T^0(X_1, X_2) = \langle JX_1, X_2 \rangle Z, \qquad \langle T^0(Y, X_1), X_2 \rangle = \langle \tau_YX_1, X_2 \rangle.$$
Write $R^0$ for the curvature of $\nabla^0$.

We can parametrize any other strongly compatible connection $\nabla = \nabla^\mu$ by a one-form $\mu$ such that
$$\nabla_Y X = \nabla^0_Y X + \mu(Y) JX, \qquad X \in \Gamma(E), Y \in \Gamma(TM).$$
Write $\mu_{-j} = \mu(\pr_{-j} \, \cdot \,)$ for $j =1,2, 3$.
\begin{theorem} \label{th:Morimoto235}
Consider the vector field $\Upsilon \in \Gamma(E)$, defined by
$$J\Upsilon = \frac{1}{4} \tr_E \varphi(\ell' [\times, J\times]  - [\times, \ell' J\times] - [ \ell' \times, J\times]  ).$$
and a one-form $\mu_{-1}(X) = \langle \Upsilon, \pr_{-1} X \rangle$.
We then have that the Morimoto grading $I = I^{W_1, W_2, \mathscr{A}}$ is given by
$$W_1 = W_2 =  \Upsilon.$$
and $\mathscr{A}$ is determined by
\begin{align*}
2 \langle \mathscr{A} X_1, X_2 \rangle_g & =  d\mu_{-1}(X_1, X_2) +2 \langle J\Upsilon, X_1\rangle_g \langle J\Upsilon ,X_2 \rangle_g\\
& \qquad - \langle J \varphi([Z', \ell' X_1]) -[Z' , JX_1] , X_2 \rangle_g .
\end{align*}
Finally, the Morimoto connection is given as $\nabla_XY =\nabla_X^0 Y + \mu(X) DY$, with
\begin{align} \label{mu-2}
4\mu_{-2}(Y)  &=  2  d\mu_{-1}(\chi(I_{-2}Y)) + \langle R^0(\chi(I_{-2}Y))- T^0(I_{-2}Y, \cdot), D\rangle_{g_I}, \\ \label{mu-3}
3\mu_{-3}(Y) & =  2 d(\mu_{-1}+\mu_{-2})(\chi(I_{-3}Y)) + \langle R^0(\chi(I_{-3}Y))- T^0(I_{-3}Y, \cdot), D\rangle_{g_I}.
\end{align}
\end{theorem}
Note that in the definition of $\mathscr{A}$, the term $d\mu_{-1}(X_1, X_2)$ only depend on the decomposition $E^{-2} = (TM)_{-1} \oplus (TM)_{-2}$, so we can hence use it to define $(TM)_{-3}$. We highlight the fact that $\Upsilon$, $W_1$, $W_2$, $\mathscr{A}$ and $\mu(\cdot)D$ do not change if we replace $Z'$ with $-Z'$ (with consequently having to replace $J$ with $-J$). They are hence defined globally.

\subsection{Proof of Theorem~\ref{th:Morimoto235}} We complete the proof in four steps. We will again write $\langle \cdot, \cdot \rangle$ for the inner product of $g_I$ and $\chi = \chi_I$ for its selector. For any horizontal unit vector field $X \in \Gamma(E)$, the selector is given by
$$\chi(Z) = X \wedge JX, \qquad \chi(\ell JX) = X \wedge Z.$$
\paragraph{\it Step 1: The $q$-map} We introduce the corresponding convenient notation to take advantage of the properties of the $I'$-grading.
For $X, Y \in \Gamma(E)$, define
$$q_X Y = \varphi([X, \ell' Y]) - \nabla_X^0 Y.$$
By the definition of $I'$, we have
\begin{equation} \label{q1} \tr_E \langle q_X \times, \times \rangle =0.\end{equation}
We observe from the Jacobi identity that
\begin{align} \label{q2}
q_X Y & = \varphi( [X, \ell' Y] ) - \nabla_X^0 Y= \varphi ([X, [Z', JY]])  - \nabla_X^0 Y \\ \nonumber
& = \varphi (-[JY,\ell' JX] - [ [X, JY], Z']) - \nabla_X^0 Y \\ \nonumber
& = - q_{JY} JX - \nabla_{JY}^0 JX - J [X,JY] - \nabla_X^0 Y  \\ \nonumber
& = - q_{JY} JX + J T^0(X, JY) = - q_{JY} JX,
\end{align}
and as a consequence
$$\tr_E q_\times \times =0.$$
Also note that for any unit vector field $X \in \Gamma(E)$
\begin{align*}
& \tr_E  q_{\times} J \times = \varphi([X, \ell' JX] - [JX, \ell' X] ) -\nabla_X^0 JX + \nabla_{JX}^0 X, \\
& = \varphi([X, \ell' JX] - [JX, \ell' X]  - \ell' [X, JX] )  \\
& = - \frac{1}{2} \tr_E \varphi(\ell' [\times, J\times]  - [\times, \ell' J\times] - [ \ell' \times, J\times]   ) = -2J \Upsilon.
\end{align*}
It follows that
$$q_X Y = q_Y X - 2\langle JX, Y\rangle J\Upsilon .$$
We further see that for any unit vector field $X \in \Gamma(E)$
\begin{align} \label{q3}
& \tr_E \langle q_X \times, J\times \rangle  = \langle q_X X, JX \rangle - \langle q_X JX, X \rangle \\ \nonumber
& = \langle q_X X, JX \rangle - \tr_{E} \langle q_\times J\times, X \rangle - \langle q_{JX} X, X \rangle \\ \nonumber
& =  \langle q_X X + q_{JX} JX, JX \rangle +  2\langle J\Upsilon, X \rangle = 2 \langle J\Upsilon, X \rangle.
\end{align}
This has the consequence that
\begin{equation} \label{q4} \langle q_X Y_1, Y_2 \rangle - \langle q_X Y_2, Y_2 \rangle = 2 \langle J\Upsilon , X \rangle \langle JY_1, Y_2 \rangle.\end{equation}

We notice also that for unit vector field $X, Y \in \Gamma(E)$, we have
$$\langle (Jq_X + q_XJ) Y, Y \rangle = - \tr_E \langle q_X \times, J \times \rangle = -  2 \langle J \Upsilon , X \rangle, $$
while 
$$\langle (Jq_X + q_XJ) Y, JY \rangle = - \tr_E \langle q_X \times,  \times \rangle =0.$$
It follows that
\begin{equation} \label{q5} Jq_X + q_X J = - 2\langle  J \Upsilon, X \rangle \id_E.\end{equation}

We finally note that
\begin{align*}
\varphi(T(X, \ell Y)) & = \nabla_X^0 Y + \mu(X) JY  -\varphi([X, \ell' Y]) - \langle JW_2, Y \rangle JX \\
& = - q_X Y  + \mu(X) JY  - \langle JW_2, Y \rangle JX 
\end{align*}

\paragraph{\it Step 2: First part of grading and connection}
Let $\theta$ be the one-form determined by $\theta(Z')  = 1$, $\ker \theta = (TM)_{-1}'\oplus (TM)_{-3}'$. Observe that
$$\langle Z, Y \rangle = \theta(Y) - \langle JW_2, \varphi(Y) \rangle.$$

We consider the restrictions of the Morimoto grading and connection. Let $X \in \Gamma(E)$ be any unit vector field. We then see that
\begin{align*}
& \langle T(\chi(Z)), X \rangle = \langle T(X, JX), X \rangle = - \mu(X) = -  \langle T_Z, \mathbb{T}_{X} \rangle \\
& = \langle T(Z, JX), Z \rangle = - \langle [Z, JX], Z \rangle = \theta([JX, Z' + JW_1]) =  - \langle JW_1, X \rangle ,
\end{align*}
giving us $\mu(X) = \langle JW_1, X \rangle$. Next,
\begin{align*}
& \langle T(\chi(\ell JX)), Z \rangle = \langle T(X, Z), Z \rangle = - \langle [X, Z], Z\rangle = - \langle JW_1, JX \rangle =  - \langle W_1, X \rangle \\
&  = - \langle T_{\ell JX} , \mathbb{T}_{Z} \rangle =  \langle T(X, \ell JX) , \ell JX \rangle - \langle T( JX, \ell JX) , \ell X \rangle \\
& =  \langle - q_{X} JX  - \mu(X) X  - \langle JW_2, JX \rangle JX  ,  JX \rangle \\
& \qquad - \langle - q_{JX} JX  - \mu(JX) X  + \langle JW_2, JX \rangle X  , X \rangle \\
& \stackrel{\eqref{q2}}{=} - \langle q_{X} JX , JX \rangle - \langle q_{X} X, X \rangle - 2  \langle W_2, X \rangle + \mu(JX) \stackrel{\eqref{q1}}{=}  \langle -2W_2 + W_1, X \rangle , \end{align*}
so $W_1 = W_2$. Finally
\begin{align*}
& 0 = \langle R(\chi(X)), D \rangle = \langle T_{X}, D \rangle \\
&  = - \langle T(X, JX), X \rangle + \langle T(X, \ell X), \ell JX \rangle - \langle T(X, \ell JX), \ell X \rangle \\
& = \mu(X)  + \langle - q_{X} X  + \mu(X) JX  - \langle JW_2, X \rangle JX, JX \rangle\\
& \qquad  - \langle - q_{X} JX  - \mu(X) X  - \langle JW_2, JX \rangle JX,  X \rangle \\
& = 3\langle  JW_1, X\rangle - \tr_E \langle q_{X} \times, J \times \rangle     - \langle JW_2, X \rangle \stackrel{\eqref{q3}}{=} \langle 2 JW_1 - 2J \Upsilon, X \rangle.
\end{align*}
We hence finally have
$$W_1 = W_2 =  \Upsilon,$$
and $\mu(X) = \mu_{-1}(X) =  \langle \Upsilon, X \rangle$ for any $X \in \Gamma(E)$.

\paragraph{\it Step 3: The complete grading}
For any pair of horizontal unit vector fields $X_1, X_2 \in \Gamma(E)$, we compute
\begin{align*}
& \langle T(\chi(\ell JX_1)), X_2 \rangle =  \langle T(X_1, Z), X_2 \rangle   \\
& = - \langle T_{\ell JX_1}, \mathbb{T}_{X_2} \rangle = - \langle T(JX_2 , \ell JX_1), Z \rangle - \langle T(Z, \ell JX_1), \ell JX_2 \rangle ,
\end{align*}
or equivalently,
$$\langle [JX_2, \ell JX_1 ], Z \rangle = \langle [Z, X_1], X_2 \rangle - \langle \varphi([Z, \ell JX_1]), JX_2 \rangle. $$
Computing these terms in detail, we have
\begin{align*}
\langle [JX_2, \ell JX_1 ], Z \rangle & = \theta([JX_2, \ell' JX_1 + \langle W_2, X_1 \rangle Z' + \mathscr{A} JX_1])  \\
& \qquad - \langle JW_2, \varphi([JX_2, \ell' JX_1 + \langle W_2, X_2 \rangle Z' + \mathscr{A} JX_1]) \rangle \\
& = JX_2 \langle W_2, X_1 \rangle - \langle \mathscr{A} JX_1, X_2 \rangle  \\
& \qquad - \langle JW_2, q_{JX_2} JX_1 + \nabla_{JX_2}^0 JX_1  \rangle + \langle JW_2, X_1\rangle \langle W_2, X_1 \rangle   \\
& = \langle \nabla_{JX_2}^0 W_2, X_1 \rangle - \langle \mathscr{A} JX_1, X_1 \rangle  + \langle JW_2, q_{X_1} X_2 \rangle \\
& \qquad  + \langle JW_2, X_2\rangle \langle W_2, X_1 \rangle   
\end{align*}
and
\begin{align*}
& \langle [Z, X_1], X_2 \rangle - \langle \varphi([Z, \ell JX_1]), JX_2 \rangle \\
& = \langle [Z' , X_1] + \nabla_{JW_1}^0  X_1 - J \nabla_{X_1}^0 W_1, X_2 \rangle \\
& \qquad - \langle \varphi([Z', \ell' JX_1]), JX_2 \rangle + \langle \mathscr{A} J X_1, X_2\rangle \\
& \qquad  - \langle \varphi([JW_1, \ell' JX_1]), JX_2 \rangle + \langle W_2, X_1\rangle \langle W_1 ,JX_2 \rangle \\
& = \langle J \varphi([Z', \ell' JX_1]) +[Z' , X_1]  - J \nabla_{X_1}^0 W_1, X_2 \rangle \\
& \qquad + \langle \mathscr{A} J X_1, X_2\rangle  - \langle q_{JW_1} JX_1, JX_2 \rangle + \langle W_2, X_1\rangle \langle W_1 ,JX_2 \rangle \\
& = \langle J \varphi([Z', \ell' JX_1]) +[Z' , X_1] , X_2 \rangle +\langle \nabla_{X_1}^0 W_1, JX_2 \rangle \\
& \qquad + \langle \mathscr{A} J X_1, X_2\rangle - \langle Jq_{X_1} W_1, X_2 \rangle + \langle W_2, X_1\rangle \langle W_1 ,JX_2 \rangle
\end{align*}
Using that $W_1 = W_2 = \Upsilon$,
\begin{align*}
2 \langle \mathscr{A} JX_1, X_2 \rangle & = \langle \nabla_{JX_2}^0 \Upsilon, X_1 \rangle   + \langle J\Upsilon, q_{X_1} X_2 \rangle  + \langle J\Upsilon, X_2\rangle \langle \Upsilon, X_1 \rangle   \\
& \qquad  - \langle J \varphi([Z', \ell' JX_1]) +[Z' , X_1] , X_2 \rangle -\langle \nabla_{X_1}^0 \Upsilon, JX_2 \rangle \\
& \qquad + \langle Jq_{X_1} \Upsilon, X_2 \rangle - \langle \Upsilon, X_1\rangle \langle \Upsilon ,JX_2 \rangle \\
& \stackrel{\eqref{q4}}{=} - d\mu_{-1}(X_1, JX_2) - \mu_{-1}(T^0(X_1, JX_2)) - \langle J \varphi([Z', \ell' JX_1]) +[Z' , X_1] , X_2 \rangle   \\
& \qquad + \langle(q_{X_1} J +  Jq_{X_1}) \Upsilon, X_2 \rangle + 2 \langle J\Upsilon, X_1 \rangle \langle \Upsilon, X_2 \rangle +2 \langle \Upsilon, X_1\rangle \langle J\Upsilon ,X_2 \rangle \\
& \stackrel{\eqref{q5}}{=} d\mu_{-1}(JX_1, X_2)  - \langle J \varphi([Z', \ell' JX_1]) +[Z' , X_1] , X_2 \rangle +2 \langle J\Upsilon, JX_1\rangle \langle J\Upsilon ,X_2 \rangle
\end{align*}
This completes the grading

\paragraph{\it Step 4: The connection} In the last step we use that
$$R(X,Y) = R^0(X, Y) + d\mu(X,Y) D.$$
Hence, for $Y$ with values in $(TM)_{-2} \oplus (TM)_{-3}$,
\begin{align*}
& \langle R(\chi(Y)), D \rangle = \langle R^0(\chi(Y)), D \rangle + \langle d\mu(\chi(Y))D, D \rangle \\
& = \langle R^0(\chi(Y)), D \rangle + 2 d\mu(\chi(Y)) = \langle T_Y, D \rangle \\
& = \langle T^0_Y, D \rangle + \langle \mu(Y)D, D \rangle - \langle \mu(\cdot) DY, D \rangle \\
& = \langle T^0_Y, D \rangle + 2 \mu(Y) -  \mu_{-3}(Y)
\end{align*}
Using that $d\mu_{-2}(\chi(I_{-2}Y)) = - \mu_{-2}(Y)$ and that $d\mu_{-3}(\chi(I_{-3}Y)) = - \mu_{-3}(Y)$, the equations in \eqref{mu-2} and \eqref{mu-3} follows.

%% file: CartanAffine.bbl
\def\cprime{$'$} \def\cprime{$'$} \def\cprime{$'$}
\begin{thebibliography}{10}

\bibitem{AgBa12}
A.~Agrachev and D.~Barilari.
\newblock Sub-{R}iemannian structures on 3{D} {L}ie groups.
\newblock {\em J. Dyn. Control Syst.}, 18(1):21--44, 2012.

\bibitem{ABB20}
A.~Agrachev, D.~Barilari, and U.~Boscain.
\newblock {\em A comprehensive introduction to sub-{R}iemannian geometry},
  volume 181 of {\em Cambridge Studies in Advanced Mathematics}.
\newblock Cambridge University Press, Cambridge, 2020.
\newblock From the Hamiltonian viewpoint, With an appendix by Igor Zelenko.

\bibitem{AMS19}
D.~Alekseevsky, A.~Medvedev, and J.~Slovak.
\newblock Constant curvature models in sub-{R}iemannian geometry.
\newblock {\em J. Geom. Phys.}, 138:241--256, 2019.

\bibitem{Alm14}
D.~M. Almeida.
\newblock Sub-{R}iemannian homogeneous spaces of {E}ngel type.
\newblock {\em J. Dyn. Control Syst.}, 20(2):149--166, 2014.

\bibitem{AmSi53}
W.~Ambrose and I.~M. Singer.
\newblock A theorem on holonomy.
\newblock {\em Trans. Amer. Math. Soc.}, 75:428--443, 1953.

\bibitem{Bel96}
A.~Bella\"{\i}che.
\newblock The tangent space in sub-{R}iemannian geometry.
\newblock In {\em Sub-{R}iemannian geometry}, volume 144 of {\em Progr. Math.},
  pages 1--78. Birkh\"{a}user, Basel, 1996.

\bibitem{BHM20}
I.~{Beschastnyi}, K.~{Habermann}, and A.~{Medvedev}.
\newblock {Cartan connections for stochastic developments on sub-Riemannian
  manifolds}.
\newblock {\em arXiv e-prints}, page arXiv:2006.16135, June 2020, 2006.16135.

\bibitem{BFG99}
P.~Bieliavsky, E.~Falbel, and C.~Gorodski.
\newblock The classification of simply-connected contact sub-{R}iemannian
  symmetric spaces.
\newblock {\em Pacific J. Math.}, 188(1):65--82, 1999.

\bibitem{Bro82}
R.~W. Brockett.
\newblock Control theory and singular {R}iemannian geometry.
\newblock In {\em New directions in applied mathematics ({C}leveland, {O}hio,
  1980)}, pages 11--27. Springer, New York-Berlin, 1982.

\bibitem{ChEb75}
J.~Cheeger and D.~G. Ebin.
\newblock {\em Comparison theorems in {R}iemannian geometry}.
\newblock North-Holland Publishing Co., Amsterdam-Oxford; American Elsevier
  Publishing Co., Inc., New York, 1975.
\newblock North-Holland Mathematical Library, Vol. 9.

\bibitem{CGJK19}
Y.~Chitour, E.~Grong, F.~Jean, and P.~Kokkonen.
\newblock Horizontal holonomy and foliated manifolds.
\newblock {\em Ann. Inst. Fourier (Grenoble)}, 69(3):1047--1086, 2019.

\bibitem{FaGo95}
E.~Falbel and C.~Gorodski.
\newblock On contact sub-{R}iemannian symmetric spaces.
\newblock {\em Ann. Sci. \'{E}cole Norm. Sup. (4)}, 28(5):571--589, 1995.

\bibitem{FaGo96}
E.~Falbel and C.~Gorodski.
\newblock Sub-{R}iemannian homogeneous spaces in dimensions {$3$} and {$4$}.
\newblock {\em Geom. Dedicata}, 62(3):227--252, 1996.

\bibitem{Gav77}
B.~Gaveau.
\newblock Principe de moindre action, propagation de la chaleur et estim\'{e}es
  sous elliptiques sur certains groupes nilpotents.
\newblock {\em Acta Math.}, 139(1-2):95--153, 1977.

\bibitem{GoGr17}
M.~Godoy~Molina and E.~Grong.
\newblock Riemannian and sub-{R}iemannian geodesic flows.
\newblock {\em J. Geom. Anal.}, 27(2):1260--1273, 2017.

\bibitem{Gro96}
M.~Gromov.
\newblock Carnot-{C}arath\'{e}odory spaces seen from within.
\newblock In {\em Sub-{R}iemannian geometry}, volume 144 of {\em Progr. Math.},
  pages 79--323. Birkh\"{a}user, Basel, 1996.

\bibitem{Gro16b}
E.~Grong.
\newblock Submersions, {H}amiltonian systems, and optimal solutions to the
  rolling manifolds problem.
\newblock {\em SIAM J. Control Optim.}, 54(2):536--566, 2016.

\bibitem{Gro20}
E.~{Grong}.
\newblock {Affine connections and curvature in sub-Riemannian geometry}.
\newblock {\em arXiv e-prints}, page arXiv:2001.03817, Jan 2020, 2001.03817.

\bibitem{Gro22b}
E.~Grong.
\newblock Curvature and the equivalence problem in sub-{R}iemannian geometry.
\newblock {\em Arch. Math. (Brno)}, 58(5):295--327, 2022.

\bibitem{Hug95}
W.~K. Hughen.
\newblock {\em The sub-{R}iemannian geometry of three-manifolds}.
\newblock ProQuest LLC, Ann Arbor, MI, 1995.
\newblock Thesis (Ph.D.)--Duke University.

\bibitem{Jea14}
F.~Jean.
\newblock {\em Control of nonholonomic systems: from sub-{R}iemannian geometry
  to motion planning}.
\newblock SpringerBriefs in Mathematics. Springer, Cham, 2014.

\bibitem{KoNo63}
S.~Kobayashi and K.~Nomizu.
\newblock {\em Foundations of differential geometry. {V}ol {I}}.
\newblock Interscience Publishers, a division of John Wiley \& Sons, New
  York-London, 1963.

\bibitem{LDO16}
E.~Le~Donne and A.~Ottazzi.
\newblock Isometries of {C}arnot groups and sub-{F}insler homogeneous
  manifolds.
\newblock {\em J. Geom. Anal.}, 26(1):330--345, 2016.

\bibitem{Mit85}
J.~Mitchell.
\newblock On {C}arnot-{C}arath\'eodory metrics.
\newblock {\em J. Differential Geom.}, 21(1):35--45, 1985.

\bibitem{Mon93}
R.~Montgomery.
\newblock Generic distributions and {L}ie algebras of vector fields.
\newblock {\em J. Differential Equations}, 103(2):387--393, 1993.

\bibitem{Mon02}
R.~Montgomery.
\newblock {\em A tour of subriemannian geometries, their geodesics and
  applications}, volume~91 of {\em Mathematical Surveys and Monographs}.
\newblock American Mathematical Society, Providence, RI, 2002.

\bibitem{Mor93}
T.~Morimoto.
\newblock Geometric structures on filtered manifolds.
\newblock {\em Hokkaido Math. J.}, 22(3):263--347, 1993.

\bibitem{Mor08}
T.~Morimoto.
\newblock Cartan connection associated with a subriemannian structure.
\newblock {\em Differential Geom. Appl.}, 26(1):75--78, 2008.

\bibitem{PTV96}
F.~Pr\"{u}fer, F.~Tricerri, and L.~Vanhecke.
\newblock Curvature invariants, differential operators and local homogeneity.
\newblock {\em Trans. Amer. Math. Soc.}, 348(11):4643--4652, 1996.

\bibitem{Sag08}
K.~Sagerschnig.
\newblock {\em Weyl structures for generic rank two distributions in dimension
  five}.
\newblock PhD thesis, uniwien, 2008.

\bibitem{Sha97}
R.~W. Sharpe.
\newblock {\em Differential geometry}, volume 166 of {\em Graduate Texts in
  Mathematics}.
\newblock Springer-Verlag, New York, 1997.
\newblock Cartan's generalization of Klein's Erlangen program, With a foreword
  by S. S. Chern.

\bibitem{Sin60}
I.~M. Singer.
\newblock Infinitesimally homogeneous spaces.
\newblock {\em Comm. Pure Appl. Math.}, 13:685--697, 1960.

\bibitem{Str86}
R.~S. Strichartz.
\newblock Sub-{R}iemannian geometry.
\newblock {\em J. Differential Geom.}, 24(2):221--263, 1986.

\bibitem{Tan76}
N.~Tanaka.
\newblock On non-degenerate real hypersurfaces, graded {L}ie algebras and
  {C}artan connections.
\newblock {\em Japan. J. Math. (N.S.)}, 2(1):131--190, 1976.

\bibitem{Tan89}
S.~Tanno.
\newblock Variational problems on contact {R}iemannian manifolds.
\newblock {\em Trans. Amer. Math. Soc.}, 314(1):349--379, 1989.

\bibitem{CaSl09}
A.~\v{C}ap and J.~Slov\'{a}k.
\newblock {\em Parabolic geometries. {I}}, volume 154 of {\em Mathematical
  Surveys and Monographs}.
\newblock American Mathematical Society, Providence, RI, 2009.
\newblock Background and general theory.

\bibitem{Web78}
S.~M. Webster.
\newblock Pseudo-{H}ermitian structures on a real hypersurface.
\newblock {\em J. Differential Geometry}, 13(1):25--41, 1978.

\end{thebibliography}
